\def\draftdate{\today}

\newif\ifedit\editfalse

\documentclass{amsart}
\usepackage{amssymb}
\usepackage{xy}
\usepackage{hyperref}

\newcommand{\bs}{\backslash}

\newcommand{\subdot}{_{\bullet}}

\let\iso\cong
\let\sma\wedge

\mathchardef\varDelta="7101
\newcommand{\DDelta}{{\mathbf \varDelta}}
\mathchardef\varLambda="7103
\newcommand{\LLambda}{{\mathbf \varLambda}}

\newcommand{\epi}{nu}
\newcommand{\Spectra}{\aS}
\newcommand{\Conf}{C}
\newcommand{\Top}{\aT}

\newcommand{\usc}{\psi}
\newcommand{\USC}{\Psi}

\newcommand{\htp}{\simeq}
\renewcommand{\to}{\mathchoice{\longrightarrow}{\rightarrow}{\rightarrow}{\rightarrow}}
\newcommand{\from}{\mathchoice{\longleftarrow}{\leftarrow}{\leftarrow}{\leftarrow}}

\newcommand{\overto}[1]{\xrightarrow{\,#1\,}}
\newcommand{\overfrom}[1]{\xleftarrow{\,#1\,}}

\DeclareMathAlphabet{\catsymbfont}{U}{rsfs}{m}{n}

\newcommand{\aD}{{\catsymbfont{D}}}
\newcommand{\aE}{{\catsymbfont{E}}}
\newcommand{\EFV}[1][{}]{\aE{mb}_{V}^{#1}}
\newcommand{\aF}{{\catsymbfont{F}}}

\newcommand{\aI}{{\catsymbfont{I}}}
\newcommand{\aJ}{{\catsymbfont{J}}}

\newcommand{\aP}{{\catsymbfont{P}}}

\newcommand{\aS}{{\catsymbfont{S}}}
\newcommand{\aT}{{\catsymbfont{T}}}

\newcommand{\bC}{{\mathbb{C}}}
\newcommand{\bD}{{\mathbb{D}}}
\newcommand{\bE}{{\mathbb{E}}}

\newcommand{\bP}{{\mathbb{P}}}
\newcommand{\bR}{{\mathbb{R}}}
\newcommand{\bS}{{\mathbb{S}}}
\newcommand{\bT}{{\mathbb{T}}}

\newcommand{\bZ}{{\mathbb{Z}}}

\newcommand{\oA}{\mathcal{A}}

\newcommand{\oC}{\mathcal{C}}
\newcommand{\Com}{\mathop{\oC\mathrm{om}}\nolimits}
\newcommand{\oD}{\mathcal{D}}
\newcommand{\oE}{\mathcal{E}}
\newcommand{\oEnd}{\oE{nd}}
\newcommand{\oH}{\mathcal{H}}

\newcommand{\oR}{\mathcal{R}}

\def\quickop#1{\expandafter\DeclareMathOperator\csname
#1\endcsname{#1}}
\quickop{id}\quickop{Id}
\quickop{Tor}\quickop{Ext}
\quickop{Hom}\quickop{Tot}
\quickop{holim}\quickop{hocolim}\quickop{hofib}
\quickop{gl}
\quickop{Emb}
\quickop{Sp}\quickop{Alg} \quickop{subgroup}
\quickop{op}\quickop{Ad}\quickop{GL}\quickop{Iso}

\numberwithin{equation}{section}

\newtheorem{thm}[equation]{Theorem}

\newtheorem{cor}[equation]{Corollary}
\newtheorem{conj}[equation]{Conjecture}

\newtheorem{prop}[equation]{Proposition}
\theoremstyle{definition}
\newtheorem{defn}[equation]{Definition}
\newtheorem{cons}[equation]{Construction}
\newtheorem{ter}[equation]{Terminology}
\newtheorem{notn}[equation]{Notation}
\newtheorem{conv}[equation]{Convention}
\newtheorem{hyp}[equation]{Hypothesis}
\theoremstyle{remark}
\newtheorem{rem}[equation]{Remark}
\newtheorem{example}[equation]{Example}

\xyoption{arrow}
\xyoption{matrix}
\xyoption{cmtip}
\SelectTips{cm}{}

\newcommand{\term}[1]{\textit{#1}}

\newdir{ >}{{}*!/-5pt/\dir{>}}

\begin{document}

\title[Norms for compact Lie groups]
{Norms for compact Lie groups in equivariant stable homotopy theory}

\author{Andrew J. Blumberg}
\address{Department of Mathematics, Columbia University, 
New York, NY \ 10027}
\email{blumberg@math.columbia.edu}

\author{Michael A. Hill}
\address{Department of Mathematics, University of California, Los
Angeles, CA \ 90095}
\email{mikehill@math.ucla.edu}

\author{Michael A. Mandell}
\address{Department of Mathematics, Indiana University,
Bloomington, IN \ 47405}
\email{mmandell@indiana.edu}

\date{\draftdate} 

\begin{abstract}
We propose a construction of an analogue of the Hill-Hopkins-Ravenel
relative norm $N_{H}^{G}$ in the context of a positive dimensional
compact Lie group $G$ and closed subgroup $H$.  We explore expected
properties of the construction.  We show that in the case when $G$ is
the circle group (the unit complex numbers), the proposed construction
here agrees with the relative norm constructed by Angeltveit,
Gerhardt, Lawson, and the authors using the cyclic bar construction.
Our construction is based on a new perspective on equivariant
factorization homology, using framings to convert from actions of one
group to another.
\end{abstract}

\maketitle


\section*{Introduction}

The work of Hill-Hopkins-Ravenel on the Kervaire invariant problem has
reinvigorated interest in the foundations of equivariant stable
homotopy theory.  In particular, the multiplicative norm construction,
a key technical tool in their work, provides a spectral version of the
Evens transfer from group cohomology.

The theory of the norm now seems to complete our understanding of the
multiplicative structure on the (genuine) equivariant stable category
when $G$ is finite: additively, the equivariant stable category is
characterized by the existence of transfers, or equivalently, the
Wirthmuller isomorphism; the multiplicative structure is similarly
characterized by the presence of compatible systems of multiplicative
norm functors.  The combinatorics of the transfers and norms are
controlled by the structure of the $G$-$E_\infty$ operad, and the
additive and multiplicative structures are linked by the
$(\Sigma^{\infty}_+, \gl_1)$ adjunction.

When $G$ is a compact Lie group, the situation is significantly more
subtle and less well-understood.  The additive structure of the
$G$-stable category is still controlled by transfers, or equivalently,
Wirthmuller isomorphisms.  However, the classical multiplicative
structures appear to be different.  Specifically, the multiplicative
structure of a commutative ring $G$-spectrum ($=G$-$E_{\infty}$ ring
spectrum) can be described in terms of an operad, but we know that the
additive structure cannot be described in this way.  Moreover, the
theory of the norm in this setting is incomplete, insofar as the
construction appears to only make sense for subgroups of finite index
in $G$.

This raises the question of what the norm for compact Lie groups
should mean in general.  A first clue is provided by the defining
adjunction for the norm on commutative ring $G$-spectra; in this
setting, the norm $N_H^G$ is the left adjoint to the forgetful functor
from $G$-spectra to $H$-spectra.  These adjoints exist and are
homotopically meaningful for any subgroup $H$ in $G$.  Therefore, we
have norm functors for at least commutative ring $G$-spectra.  To explore this
further, we consider the simplest possible example: $G = S^1$ and the
subgroup is the trivial subgroup.
In this case, we recover a familiar object:
$N_e^{S^1} R$ is precisely $THH(R)$ \cite{ABGHLM}.  This description
now makes sense not just for commutative ring spectra but more
generally for associative ring spectra. 

While $THH(R)$ admits several constructions, its identification in
terms of factorization homology of $S^{1}$ with values in $R$ suggests
an approach to the construction of more general norms using
factorization homology. In the case of norms of the form $N_{e}^{G}$,
it suggests a construction in terms of a genuine equivariant structure
on the factorization homology of $G$; it also indicates the type of
algebras that could admit such a norm.  The construction of a norm
$N_H^G$ requires an equivariant version of a factorization homology
construction.

The purpose of this paper is to propose and outline an approach to
equivariant norms in terms of factorization homology.  Specifically,
we describe a factorization homology style construction of positive
dimensional norms.  We discuss aspects of its expected homotopy type
in terms of geometric fixed point data and lay out a
series of conjectures about the properties of such norms.
In the case of the circle group, we show that the proposed definition
here agrees with the definition in~\cite{ABGHLM}; the conjectures that
pertain to this context are theorems, proved \textit{ibid}.
Moreover, whereas \cite{ABGHLM} only treated the
homotopy type of $THH(R)$ as well-defined in the $\aF$-model
structure (where weak equivalences are detected on passage to fixed
points for finite subgroups of $G$), the work here identifies the full
genuine equivariant homotopy type in the context of the norm.

The factorization homology we use to construct the norm is not
equivariant factorization homology as it is usually construed.  Our
use of factorization homology and its design in Section~\ref{sec:fhne}
has the goal of using $G$-equivariant $H$-framed manifolds to mediate
a conversion of $H$-equivariant orthogonal spectra (with extra
structure) to $G$-equivariant orthogonal spectra.  We are not
constructing a general theory of $G$-equivariant factorization
homology here.  Nevertheless, our construction appears to reproduce
the standard versions of genuine $G$-equivariant factorization
homology for $V$-framed $G$-manifolds in the literature.  To
illustrate this, we show how to obtain a version of $G$-equivariant
factorization homology of $V$-framed $G$-manifolds in
Section~\ref{sec:efh} as a special case of the theory of
Section~\ref{sec:fhne}.

The work here on the norm depends on current work in progress of some
subsets of the authors
which breaks into two projects.  The first, which we cite as [CFH]
studies non-equivariant factorization homology from the perspective of
making as much of the structure as possible covariant with compact Lie
group actions.  A mostly complete draft exists but is not publicly
available; we state the results that depend on it here as
\textbf{theorems} (but we acknowledge that they should more properly
be labeled as \textbf{conjectures}).  The second project, which we
refer to as [PMI] will study generalizations of the norm construction
for finite groups of the form
\[
I'(A \sma_{\Sigma_{q}\wr H}IX^{(q)})
\]
where is a genuine equivariant $H$-spectrum, $A$ is a $G\times
(\Sigma_{q}\wr H)$-space, and $I,I'$ are certain point-set change of universe
functors (for some compact Lie groups $H,G$). In the case when $H$ is a
subgroup of a finite group $G$ and $A$ is the set of (numbered) coset
representatives 
of $H$ in $G$ (plus a disjoint base point), this construction is precisely the norm $N_{H}^{G}X$,
but for other $H$, $G$, and $A$, it produces more general functors
from $H$-spectra to $G$-spectra.  The purpose of [PMI] is to study
the equivariant homotopy theory of such functors and specifically
to verify the expected formulas for geometric fixed point spectra
(generalizing the norm diagonal formulas).  A basic result we will use
is the following generalization of \cite[B.104,B.146]{HHR}:
\par\smallskip
{\narrower\noindent
Let $G$ be a compact Lie group, $A$ a $\Sigma_{q}$-free $G\times
\Sigma_{q}$-CW complex, $U$ a complete $G$-universe, $X'$ a cofibrant
orthogonal spectrum and $X$ either a cofibrant orthogonal spectrum or
a cofibrant commutative ring orthogonal spectrum.  If $X'\to X$ is a
weak equivalence of orthogonal spectra, then
\[
I_{U^{G}}^{U}(A\sma_{\Sigma_{q}}X^{\prime(q)})\to 
I_{U^{G}}^{U}(A\sma_{\Sigma_{q}}X^{(q)})
\]
is a weak equivalence of orthogonal $G$-spectra indexed on $U$.\par}
\smallskip
\noindent
Most of the work of [PMI] beyond the statement above is preliminary,
and anything discussed below that depends on more complicated
results or more complicated constructions of this type is labeled as a
``conjecture''. 

Only the work in Sections~\ref{sec:absolute}--\ref{sec:rel} depend on
the unpublished work [CFH] and [PMI].  Sections~\ref{sec:circle}
and~\ref{sec:efh} depend only on Construction~\ref{cons:compressed};
they are completely independent of the work in progress and
non-conjectural.

\subsection*{Acknowledgments}
The authors were supported in part by NSF grants DMS-2104348,
DMS-2104420, and DMS-2105019.
This material is based upon work supported by the National Science
Foundation under Grant No.~1440140, while some of the authors were in
residence at the Mathematical Sciences Research Institute in Berkeley,
California, during the Fall semester of 2022.  The authors would like
to thank Mike Hopkins as well as their frequent
collaborators Vigleik Angeltveit, Teena Gerhardt, and Tyler Lawson for
many useful conversations on these topics spanning many years.
The authors also thank Inbar Klang and Charles Rezk for helpful comments.

\section{Universes in equivariant stable homotopy theory}\label{sec:univ}

The purpose of this section is to review the role that ``universes''
play in the orthogonal spectrum models of equivariant stable homotopy
theory.  Although constructions like the Hill-Hopkins-Ravenel norm
were anticipated for many years~\cite[\S3--4]{GreenleesMay-MUCompletion},
\cite[1.5]{NielsenRognes-Singer}, it took a surprisingly long time for 
the precise definition to appear because describing the correct
derived functor requires playing off two distinct perspectives on the
role of the universe.  This phenomenon permeates our work here, and we
use this section to carefully explain the situation and our
perspective and terminology.  We make no particular claim to
originality in this section.

Let $G$ be a compact Lie group.  A $G$-universe is a countably infinite
dimensional vector space with linear $G$-action and a $G$-invariant
inner product satisfying the following properties: (i) $U$ contains a
copy of the trivial representation, and (ii) if a given finite dimensional
representation occurs in $U$ as a $G$-stable vector subspace, then $U$
contains a countable direct sum of copies of that representation.  The
inner product space $\bR^{\infty}=\bigcup \bR^{n}$ with the usual
inner product and trivial action is a $G$-universe, and up to
isomorphism, every $G$-universe contains it as a sub inner product
space.  We use the notation $V<U$ to mean that $V$ is a $G$-stable
finite dimensional vector subspace of $U$ (which then inherits the
structure of a $G$-invariant inner product space), and for $W<U$, we write
$V<W$ to mean that $V$ is a $G$-stable vector subspace of $W$ (not
necessarily proper).  In that case, we write $W-V$ for the orthogonal
complement of $V$ in $W$, and for any $V<U$, we write $S^{V}$ for the
one-point compactification of $V$ and $\Sigma^{V}(-)$ for the
$V$-suspension $(-)\sma S^{V}$.

The classical view is that given a universe, we get a category of
spectra (or ``prespectra'' for some authors) \term{indexed on $U$}:
a spectrum $T$ indexed on $U$ consists of a
based $G$-space $T(V)$ for every $V<U$ and a structure $G$-map
$\sigma_{V,W}\colon \Sigma^{W-V}T(V)\to T(W)$ for every $V<W<U$ such
that $\sigma_{V,V}=\id$ and when $V<W<X<U$, $\Sigma_{V,X}=\Sigma_{W,X}\circ
\Sigma^{X-W}\Sigma_{V,W}$, that is, the diagram
\[
\xymatrix@R-1pc@C+2pc{%
\Sigma^{X-W}\Sigma^{W-V}T(V)\ar[r]^-{\Sigma^{X-W}\sigma_{V,W}}
&\Sigma^{X-W}T(W)\ar[d]^-{\sigma_{W,X}}\\
\Sigma^{X-V}T(V)\ar[u]^-{\iso}\ar[r]_-{\sigma_{V,X}}
&T(X)\\
}
\]
commutes. We then define homotopy groups and weak equivalences (for
example as in~\cite[III\S3]{MM}), and
inverting these weak equivalences gives the $U$-universe $G$-equivariant
stable category.  In fact, the $U$-universe $G$-equivariant stable
category is the homotopy category of a ($G$-topological) model
structure on the category of spectra indexed on $U$; we do not need
the details but they can be found for example in~\cite[III\S4]{MM}.  

Various universes lead to different equivariant stable categories.
The role of the universe is to provide $G$-vector spaces that compact
$G$-manifolds $M$ (e.g., orbits) can embed in; which orbits
equivariantly embed in the universe precisely determine the
equivariant stable homotopy theory.  Put another way, the universe
controls which equivariant transfers exist.  In the two extreme cases,
when $U$ contains only the trivial representation (a \term{trivial
universe}), e.g., $U=\bR^{\infty}$, and when $U$ contains every
representation (a \term{complete universe}), the equivariant stable
categories are inequivalent whenever $G$ is not the trivial (one
point) group.  The $U$-universe equivariant stable category is often
called the \term{naive equivariant stable category} when $U$ is a
trivial universe and the \term{genuine equivariant stable category}
when $U$ is a complete universe.  More generally, the adjective
\term{naive} means indexed on a trivial universe and the adjective
\term{genuine} means indexed on a complete universe.

The same approach applies to $G$-equivariant orthogonal spectra.
Given a $G$-universe $U$, we define a $G$-equivariant
orthogonal spectra indexed on $U$ to consist of:
\begin{enumerate}
\item For every $V<U$, a based $G$-space $T(V)$;
\item For every $V,W<U$, a $G$-map $\alpha_{V,W}\colon
O(V,W)_{+}\sma T(V)\to T(W)$, where $O(V,W)$ denotes the
$G$-space of (non-equivariant) isometric isomorphisms from $V$ to $W$;
and  
\item For every $V<W<U$, a $G$-map $\sigma_{V,W}\colon
\Sigma^{W-V}T(V)\to T(W)$ 
\end{enumerate}
satisfying the obvious compatibility relations: the data of (ii) make
$T$ a $G$-topo\-logically enriched functor on the $G$-topological
category with objects $V<U$ and maps $O(-,-)$, the data of (iii)
makes $T$ into a spectrum indexed on $U$, and the maps in (iii) are
$O(V,V)\times O(W,W)$-equivariant.  (This formulation is
slightly different from the formulation in the standard
reference~\cite[II\S2]{MM}, but gives an equivalent category.)
We define the weak equivalences to be the weak equivalences of the
underlying $G$-equivariant prespectra indexed on $U$; inverting these
weak equivalences, the forgetful functor to the $U$-universe
$G$-equivariant stable category is an equivalence. Thus, we can treat
the $U$-universe $G$-equivariant stable category as the localization
of the category of $G$-equivariant orthogonal spectra indexed on $U$
at its natural weak equivalences. 

The category of $G$-equivariant orthogonal spectra indexed
on $U$ has a ($G$-topo\-logical) model structure with fibrations and
weak equivalences created by the forgetful functor to $G$-equivariant
spectra indexed on $U$.  We use the existence of this structure, but
not the details, which can be found for example in~\cite[III\S4]{MM}.

The theory of $G$-equivariant orthogonal spectra admits another
interpretation first articulated by Mandell-May in~\cite[V\S1]{MM} and
popularized by Schwede in~\cite[\S2]{Schwede-EquivariantNotes}.  Let
$\Spectra$ denote the category of (non-equivariant) orthogonal spectra
indexed on $\bR^{\infty}$, and let $\Spectra^{G}_{U}$ be the category
of $G$-equivariant orthogonal spectra indexed on $U$.  Let
$\Spectra^{BG}$ be the category of $G$-objects in $\Spectra$: an
object in $\Spectra^{BG}$ consists of an object $T$ of $\Spectra$ and
an associative and unital action map $G_{+}\sma T\to T$; then
$\Spectra^{BG}$ and $\Spectra^{G}_{\bR^{\infty}}$ are essentially the
same categories and are certainly at least canonically isomorphic.  The observation
of~\cite[V.1.5]{MM} is that even though the point-set categories
$\Spectra^{G}_{U}$ and $\Spectra^{BG}$ are not isomorphic for a
non-trivial universe $U$, they are equivalent: given a spectrum $T$
indexed on $\bR^{\infty}$ with $G$-action and any $n$-dimensional
$G$-inner product space $V$, let
\[
T(V)=O(\bR^{n},V)_{+}\sma_{O(n)}T(\bR^{n}).
\]
Then $T(V)$ is a based $G$-space and the collection $\{T(V)\mid V<U\}$ has
the canonical structure of a $G$-equivariant orthogonal spectrum
indexed on $U$; we write $I_{\bR^{\infty}}^{U}T$ for this object.  We
can also go the other way, if $T$ is a $G$-equivariant orthogonal
spectrum indexed on $U$ and $V$ is any $n$-dimensional $G$-stable
subspace of $U$, the $G$-spaces
\[
O(V,\bR^{n})_{+}\sma_{O(V)}T(V)
\]
(where $O(V)=O(V,V)$) are all canonically isomorphic: any
non-equivariant isometric isomorphism $f\colon V\to V'$ induces the
same isomorphism
\[
O(V,\bR^{n})_{+}\sma_{O(V)}T(V)\to O(V',\bR^{n})_{+}\sma_{O(V')}T(V').
\]
This gives a functor $I_{U}^{\bR^{\infty}}$ from $\Spectra^{G}_{U}$ to
$\Spectra^{G}_{\bR^{\infty}}$ or $\Spectra^{BG}$.  The functors
$I_{\bR^{\infty}}^{U}$ and $I_{U}^{\bR^{\infty}}$ are inverse
equivalences of categories.  More generally, for any pair of
$G$-universes $U$, $U'$, formulas of this type define inverse
equivalences $I_{U}^{U'}$, $I_{U'}^{U}$ between the point-set
categories $\Spectra^{G}_{U}$ and $\Spectra^{G}_{U'}$.  We call these
\term{point-set change of universe functors}.

Since the point-set change of universe functor $I^{\bR^{\infty}}_{U}$
is an equivalences of categories, we can use it to transport the model
structure on $G$-equivariant orthogonal spectra indexed on $U$
($\Spectra^{G}_{U})$ to $G$-equivariant orthogonal spectra indexed on
$\bR^{\infty}$ ($\Spectra^{G}_{\bR^{\infty}}$) or $G$-objects in
orthogonal spectra $\Spectra^{BG}$.  If $U$ is a trivial universe,
then this agrees with the intrinsic model structure on
$\Spectra^{G}_{\bR^{\infty}}$; if $U$ is non-trivial, then this model
structure is different and the homotopy category is inequivalent.
When we view $\Spectra^{G}_{\bR^{\infty}}$ as $G$-objects in
orthogonal spectra ($\Spectra^{BG}$), none of these model structures
are the ``usual'' model structure, which would have as its weak
equivalences the weak equivalences of the underlying orthogonal
spectra. (These are commonly called the \term{Borel equivalences} and
they are the weak equivalences of a \term{Borel model structure}; we
never use these weak equivalences or this model structure in this paper.)

Each universe $U$ therefore produces on $\Spectra^{BG}$ a model
structure that we call the \term{$U$-universe model structure}. We
call its weak equivalences the \term{$U$-universe} weak equivalences.
The $U$-universe model structure on $\Spectra^{BG}$ is Quillen
equivalent (via $I_{\bR^{\infty}}^{U}$ and $I_{U}^{\bR^{\infty}}$) to
the category of $G$-equivariant orthogonal spectra indexed on $U$, and
in particular the homotopy category of the $U$-universe model
structure is equivalent to the $U$-universe $G$-equivariant stable
category.  One can take the perspective then that there is (up to
equivalence) only one point-set category of $G$-equivariant orthogonal
spectra and that the role of universes is to define the weak
equivalences, or more rigidly, the model structure.

The construction of norms intrinsically uses the perspective that the
point-set category of $G$-equivariant orthogonal spectra indexed on a
complete universe is just the category of $G$-objects in orthogonal
spectra.  This paper contains a number of point-set constructions that
only make sense for $G$-objects in orthogonal spectra but that we
argue (or conjecture) are homotopically meaningful in $U$-universe
model structures for a complete universe $U$.  We have chosen to
indicate this by using point-set change of universe functors to
specify inside a point-set construction what our homotopical
perspective on the weak equivalences is.  We start in spectra indexed
on a universe $U$, use the point-set change of universe functor
$I_{U}^{\bR^{\infty}}$ to $\bR^{\infty}$, do a point-set construction,
and finally do the point-set change of universe functor
$I_{\bR^{\infty}}^{U}$ to index on a universe $U'$; this indicates
that we expect the overall construction to convert $U$-universe weak 
equivalences to $U'$-universe weak equivalences at least for nice
(e.g., $U$-universe cofibrant) input.

\begin{rem}
Although the point-set change of universe functor
$I_{U}^{\bR^{\infty}}$ admits a right derived functor, the
construction of the Hill-Hopkins-Ravenel norm (and our constructions
here) use it in a rhetorical and non-homotopical way.  As such, it is
amazing that the overall construction results in a functor that has a
homotopical interpretation, admitting a left derived functor.
\end{rem}

\section{The absolute case}\label{sec:absolute}

We begin with the absolute case of the norm $N_{e}^{G}X$ which takes
non-equivariant input and produces a genuine equivariant
$G$-spectrum.  Fix a compact Lie group $G$ with manifold dimension $d$
and fix a complete $G$-universe $U$ satisfying $U^{G}=\bR^{\infty}$.
We also fix a basis of the tangent space of $G$ at the identity and
use left invariant vector fields to specify a basis for the tangent
space at every point, specifying a parallelization of $G$.  The left
action of $G$ on itself is then through maps of parallelized manifolds
and embeds $G$ as a subgroup of the topological group of automorphisms
of $G$ as a parallelized manifold (see Example~\ref{ex:G}).

We assume a continuous point-set factorization homology functor
\[
\int\colon \aE_{e}\times \Spectra[\oD^{d}] \to \Spectra,
\]
where $\aE_{e}$ is the (topological) category of parallelized
$d$-manifolds (with maps the parallelized embeddings; see Definition~\ref{defn:framemap}),
$\Spectra$ is (as indicated in Section~\ref{sec:univ}) the
(topological) category of orthogonal spectra, $\oD^{d}$ is the 
Boardman-Vogt little $d$-disk operad, and $\Spectra[\oD^{d}]$ is the
(topological) category of $\oD^{d}$-algebras in
$\Spectra$.  Given such a functor, we can 
make the following point-set definition.  In it, (as indicated in
Section~\ref{sec:univ}) $\Spectra^{G}_{U}$ denotes the category of
$G$-equivariant orthogonal spectra indexed on the universe $U$.

\begin{defn}\label{defn:NeG}
Define a point-set functor $N_{e}^{G}\colon \Spectra[\oD^{d}]\to
\Spectra^{G}_{U}$ by 
\[
N_{e}^{G} X = I_{U^{G}}^{U}\int_{G}X
\]
where we regard $\int_{G}X$ as a $G$-object in $\Spectra$ by the
natural left $G$-action on $G$.  As indicated in
Section~\ref{sec:univ}, $I_{U^{G}}^{U}$ denotes the point-set change
of universe functor from $\bR^{\infty}=U^{G}$ to $U$. 
\end{defn}

Definition~\ref{defn:NeG} is a point-set definition that depends
inherently on the point-set construction $\int$. It is not clear nor
do we claim that this has any homotopical properties for an arbitrary
functor $\int$.  However, bar constructions for factorization homology
tend to have good (non-equivariant) homotopical properties that can be
leveraged to study the homotopical properties of the construction
$N_{e}^{G}$.  In particular, the construction we consider in [CFH]
results in a functor $N_{e}^{G}$ that has many of the properties
expected of an equivariant norm, as we now begin to explain.

Before starting an in depth discussion, we need to address a
particular subtlety that arises in the equivariant theory for compact
Lie groups that does not arise in the non-equivariant theory or the
equivariant theory for finite groups.  We can work with unital
$\oD^{d}$-algebras or non-unital $\oD^{d}$-algebras. Standard constructions of
factorization homology can take as input a unital $\oD^{d}$-algebra or a
non-unital $\oD^{d}$-algebra (an algebra over the non-unital little $d$-disk
operad, where we replace $\oD^{d}(0)=0$ with $\oD^{d}(0)=\emptyset$). For
a unital $\oD^{d}$-algebra $X$, there is a natural map from the non-unital
construction $\int^{nu}_{M}X$ to the unital construction $\int_{M}X$
that is always a homotopy equivalence non-equivariantly (but the
homotopy inverse and homotopy data are not natural in $M$).  This map
is generally not an equivariant weak equivalence for $G$ positive
dimensional: when $X$ is the sphere spectrum $\bS$, $\int_{M}\bS$ is
$G$-equivariantly homotopy equivalent to the sphere spectrum, but
$\int^{nu}_{M}\bS$ is $G$-equivariantly homotopy equivalent to
$\Sigma^{\infty}_{+}\mathrm{Ran}(M)$ for the Ran space of $M$ (the
space of finite non-empty subsets of $M$).  Non-equivariantly, the Ran
space is contractible and this is a model of $\bS$, but equivariantly,
when we take $M=G$ (as for the norm above) and we take $H<G$ a positive
dimensional closed subgroup, the geometric fixed points satisfy
\[
(\Sigma^{\infty}_{+}\mathrm{Ran}(G))^{\Phi H}\iso 
\Sigma^{\infty}_{+}(\mathrm{Ran}(G)^{H})=*.
\]
This is the wrong answer because factorization homology should take
smash products to smash products in the $X$ variable, and so the empty
smash product $\bS$ should go to the empty smash product $\bS$.
But this essentially the only thing that goes wrong: technology of
[BHM1] appears sufficient to prove that (under mild hypotheses on $X$,
e.g., the inclusion of the unit $\bS\to X$ is a Hurewicz cofibration)
the unital construction fits into a homotopy pushout square
\[
\xymatrix@-1pc{%
\int^{nu}_{M}\bS\ar[r]\ar[d]&\int^{nu}X\ar[d]\\
\hspace{-2em}
\bS\simeq \int_{M} \bS\ar[r]&\int_{M} X.
}
\]

Ayala and Francis~\cite[2.1.4]{AyalaFrancis-Poincare} define a
filtration on non-unital factorization homology that they call the ``cardinality
filtration'' and they identify the homotopy type of the cofiber in
filtration level $q$ as 
\[
Fil^{q}\int_{M}^{nu}X \bigg/ Fil^{q-1}\int_{M}^{nu}X 
\simeq \Conf^{\epi}(q,M)^{+}\sma_{\Sigma_{q}}X^{(q)}
\]
(in the case of a compact parallelized manifold $M$), where $\Conf^{\epi}$
denotes the subspace of components of the configuration space
where at least one point of the configuration
lies in each component of $M$.  Here $(-)^{+}$ denotes one-point
compactification, and $(-)^{(q)}$ denotes $q$th smash power. In
[CFH], the authors construct a point-set version of this map
with enough naturality in the map and the
homotopies\footnote{Specifically, we can make naturality work for 
a compact Lie group of automorphisms, and we claim no more generality
than that.} that in the case
of $N_{e}^{G}X$, we get an equivariant homotopy equivalence
\[
Fil^{q}\int_{G}^{nu}X \bigg/ Fil^{q-1}\int_{G}^{nu}X \simeq \Conf^{\epi}(q,G)^{+}\sma_{\Sigma_{q}}X^{(q)}.
\]
The point-set change of universe functor preserves $G$-homotopy
equivalences. Moreover, as indicated in the introduction, in 
the case when the underlying orthogonal spectrum of $X$ is cofibrant,
the genuine $G$-equivariant homotopy type of 
\[
I_{U^{G}}^{U}(\Conf(q,G)^{+}\sma_{\Sigma_{q}}X^{(q)})
\]
is invariant under weak equivalences in $X$; this happens in
particular in the case when $X$ is cofibrant as a
$\oD^{d}$-algebra and the result holds also in the case when $X$ is a
cofibrant commutative ring orthogonal spectrum.  This gives the following
result. 

\begin{thm}\label{thm:hi}
Let $X'$ be a $\oD^{d}$-algebra whose underlying orthogonal spectrum
is cofibrant (e.g., cofibrant $\oD^{d}$-algebras) and let $X$ be
either a $\oD^{d}$-algebra whose underlying orthogonal spectrum is
cofibrant or a cofibrant commutative ring orthogonal spectrum.  For
the point-set construction of $\int$ in [CFH], a weak
equivalence $X'\to X$ induces a weak equivalence $N_{e}^{G}X'\to
N_{e}^{G}X$.  
\end{thm}

\begin{cor}\label{cor:hi}
The left derived functor of $N_{e}^{G}$ exists.  Moreover, the
composite with the derived functor of the forgetful functor
$\iota\colon \Com\to \Spectra[\oD^{d}]$ from commutative ring
orthogonal spectra to $\oD^{d}$-algebras is the derived
functor of the composite $N_{e}^{G}\iota$.
\end{cor}

In the case when $X=R$ is a cofibrant commutative algebra, we have
another interpretation of the homotopy type of $N_e^G R$, which aligns
with the finite theory.  To start, the category of commutative ring
orthogonal spectra admits indexed colimits over topological spaces; in
particular this means that for a space $M$, we have a commutative
orthogonal spectrum $R\otimes M$ with the universal property that the
space of maps from $R\otimes M$ to a commutative ring orthogonal
spectrum $A$ is homeomorphic to the space of maps from $M$ to the
space of maps from $R$ to $A$ 
\[
\Top(M,\Spectra[\Com](R,A))\iso \Spectra[\Com](R\otimes M,A).
\]
In the case $M=G$, this means that
$(-)\otimes G$ is the free functor from the category of commutative
ring orthogonal spectra to the category of left $G$-objects in
commutative ring orthogonal spectra.  The point-set change of universe
$I_{U^{G}}^{U}$ is the free functor from the category of left
$G$-objects in commutative ring orthogonal spectra to the category of
commutative ring orthogonal $G$-spectra indexed on $U$.  Thus,
$I_{U^{G}}^{U}(-\otimes G)$ is the free functor from the category of
commutative ring orthogonal spectra to commutative ring orthogonal
$G$-spectra indexed on $U$ (and as such it is clearly a Quillen left
adjoint, and in particular preserves weak equivalences between
cofibrant objects). In the case when $G$ is a finite group, the tensor
$R\otimes G$ is just the smash power of $R$ indexed on the elements of
$G$ and the change of universe $I_{U^{G}}^{U}(R\otimes G)$ is
precisely the norm $N_{e}^{G}R$ \cite[2.17--18]{ABGHLM},
\cite[A.52,A.56]{HHR}.

The tensor $R\otimes G$ also has a cardinality
filtration, but in this case the associated graded cofibers look like 
\[
Fil^{q}(R\otimes G)\big/ Fil^{q-1}(R\otimes G) \simeq \Conf(q,G)^{+}\sma_{\Sigma_{q}}(R/\bS)^{(q)}.
\]
If the unit map $\bS\to R$ is a Hurewicz cofibration of the underlying
orthogonal spectra, then the maps in the filtration are Hurewicz
cofibrations and the display above is a $G$-equivariant homotopy
equivalence.  The construction of $\int$ in [CFH] admits a
corresponding filtration and a filtration preserving map $\int_{G}R\to
R\otimes G$, which induces a homotopy equivalence on associated graded
cofibers; the constructions are natural enough that the comparison is
$G$-equivariant, but in this case we have less sharp results on the
covariance of the homotopies and do not know that we can make the
homotopies $G$-equivariant.  This is not good enough to directly
conclude that the map after point-set change of universe
$I_{U^{G}}^{U}$ is a weak equivalence.  Nevertheless, the goal of
[PMI] is to study geometric fixed points of constructions of the form
\[
I_{U^{G}}^{U}(\Conf(q,G)^{+}\sma_{\Sigma_{q}}X^{(q)}).
\]
Let $K<G$ be a closed subgroup.  We note that $\Conf(q,G)^{K}$ is
empty unless $K$ is a finite group whose cardinality divides $q$, in
which case we have a canonical identification 
\[
\Conf(q,X)^{K}\iso \Conf(q/\#K,K \bs G)
\]
(where $K\bs G$ is the left quotient orbit space).
We therefore expect that for reasonable $X$, we will have 
\[
(I_{U^{G}}^{U}(\Conf(q,G)^{+}\sma_{\Sigma_{q}}X^{(q)}))^{\Phi K}\simeq *
\]
if $K$ is positive dimensional or $K$ is finite and $\#K$ does not
divide $q$, and
\[
(I_{U^{G}}^{U}(\Conf(q,G)^{+}\sma_{\Sigma_{q}}X^{(q)}))^{\Phi K}
\simeq \Conf(p,G)^{+}\sma_{\Sigma_{p}}X^{(p)}
\]
when $p=q/\#K$ is an integer.  A careful study of the comparison map
$N_{e}^{G}R\to I_{U^{G}}^{U}(R\otimes G)$ should prove that it is a
homotopy equivalence on each geometric fixed point spectrum, which
would then prove the following as a result.

\begin{conj}
Let $R$ be a cofibrant commutative algebra.  Then there is a
natural weak equivalence of genuine $G$-spectra 
\[
N_e^G R \htp I_{U^{G}}^{U}(R \otimes G)
\]
where $I_{U^{G}}^{U}(- \otimes G)$ is the left adjoint of the
forgetful functor from commutative ring orthogonal $G$-spectra indexed
on $U$ to commutative ring spectra.
\end{conj}

We can use the same ideas as in the discussion preceding the previous
conjecture to study the geometric fixed points of the construction
$N_{e}^{G}$ for more general $\oD^{d}$-algebras.  In the case when
$K<G$ is positive dimensional, the following should hold.

\begin{conj}\label{conj:geopos}
Let $K<G$ be a closed subgroup of positive dimension.  If $X$ is a
$\oD^{d}$-algebra whose underlying orthogonal spectrum is cofibrant or
$X$ is a cofibrant commutative ring orthogonal spectrum, then the inclusion
of the unit $\bS\to X$ induces a weak equivalence of derived geometric
$K$-fixed point spectra
\[
(N_{e}^{G}X)^{\Phi K}\simeq \bS.
\]
\end{conj}

The previous result gives further justification for the identification
of $N_{e}^{G}$ as a norm, as its analogue is known to hold for the functor
$I_{U^{G}}^{U}(R\otimes G)$. 

\begin{thm}[Hill-Hopkins~{\cite{HillHopkins-EqSymMon}}]
Let $R$ be a cofibrant commutative ring orthogonal spectrum, and let
$K<G$ be a closed subgroup of positive dimension.  Then the inclusion
of the unit $\bS$ in $R$ induces a weak equivalence of derived
geometric $K$-fixed point spectra
\[
(I_{U^{G}}^{U}(R\otimes G))^{\Phi K}\simeq \bS
\]
\end{thm}

\begin{proof}
(Compare~\cite[8.5]{HillHopkins-EqSymMon}.)  For notational
convenience, let $A=I_{U^{G}}^{U}(R\otimes G)$.  Without loss of
generality, we can assume that $R$ is a cell commutative orthogonal
spectrum; then $A$ is a cell commutative $G$-spectrum built from
attaching a commutative ring cell of the form
\[
\bP_{G} (G_{+}\sma \Sigma^{\infty}_{\bR^{m}}(B^{n},\partial
B^{n})_{+})
\]
for each cell of the form
\[
\bP (\Sigma^{\infty}_{\bR^{m}}(B^{n},\partial B^{n})_{+})
\]
building $R$ (where $\bP$ and $\bP_{G}$ are the free functors from
orthogonal spectra to commutative ring orthogonal spectra and from
orthogonal $G$-spectra to commutative ring orthogonal $G$-spectra,
respectively, and $\Sigma^{\infty}_{\bR^{m}}$ denotes the left adjoint
to the $\bR^{m}$ space functor in either category).  The underlying
orthogonal $G$-spectrum of $A$ is then built from $\bS$ by attaching
orthogonal spectrum cells of the form
\[
C=G^{q}\sma_{H} \Sigma^{\infty}_{V}(B(W),\partial B(W))_{+}
\]
where $q>0$, $H$ is a subgroup of $\Sigma_{q}$, $V$ and $W$ are finite
dimensional inner product spaces with orthogonal $H$-actions, and
$B(W)$ denotes the unit ball (or more naturally, $B(W)$ is a polydisk
$B^{n_{1}}\times \dotsb \times B^{n_{q}}$ with $H$ acting by permuting
the factors). Precisely, for $Z\subset U$ a finite dimensional
$G$-stable subspace, the $Z$ space pair of the orthogonal $G$-spectrum
pair $C$ is
\[
C(Z)=G^{q}_{+}\sma_{H} (\aJ(V,Z)\sma (B(W),\partial B(W))_{+})
\]
(where $\aJ$ is the Thom space of~\cite[II.4.1]{MM}).  Because $H$ is
finite and $K$ is not, neither space in the pair $C(Z)$ can have any
$K$-fixed points other than the base point, and so both genuine
$G$-spectra in the pair $C$ have trivial derived geometric $K$ fixed
point spectra, for example, by~\cite[V.4.8.(ii), V.4.8.12]{MM}.  It
follows that cell attachment by $C$ does not change derived geometric
$K$-fixed points, and that the map $\bS\to A$ induces an equivalence
on derived geometric $K$-fixed points.
\end{proof}

Given the discussion of geometric fixed points above, we would expect
that for $K$ finite,
\[
(R\otimes G)^{\Phi K} \simeq R\otimes (K\bs G)
\]
as an $NK/K$-spectrum. The analogous formula
\[
(N_{e}^{G}X)^{\Phi K} \simeq \int_{K\bs G} X
\]
does not always make sense because $K\bs G$ is not always
parallelizable: $K\bs G$ inherits a parallelization from $G$ if and
only if the $\Ad$ action of $K$ on the tangent space of $G$ at the
identity $T_{e}G$ is trivial.  In this case, we make the following conjecture.

\begin{conj}
Let $K<G$ be a finite subgroup and assume the $\Ad$ action of $K$ on
$T_{e}G$ is trivial.  For composite of derived functors
$\Phi^{K}N_{e}^{G}$, there is a natural weak equivalence in the
non-equivariant stable category 
\[
(N_{e}^{G}X)^{\Phi K}\simeq \int_{K\bs G}X.
\]
\end{conj}

If $K$ is normal, then the $\Ad$ action is automatically trivial, and
we can identify the righthand side as $N_{e}^{G/K} X$; in this case, we
further conjecture that the above weak equivalence refines to a weak
equivalence in the stable category of genuine $G/K$-spectra.  When
$K\bs G$ is not parallelizable, a $\oD^{d}$-algebra does not have
enough structure to make sense of $\int_{K\bs G}$. Instead, the weak
equivalence should take the following form.

\begin{conj}
Let $K<G$ be finite.  There is a natural non-equivariant weak equivalence
\[
(N_{e}^{G}X)^{\Phi K}\simeq 
\int_{N\bs G}\int_{K\bs N\times T_{e}(N\bs G)}X
\]
where $N$ is the normalizer of $K$ in $G$.
\end{conj}

Interpreting the formula in the preceding conjecture takes some work.
The manifold $N\bs G$ is not generally parallelizable but its tangent
bundle admits a canonical reduction of structure to $N$ for the
$N$-representation given by the action on $T_{e}(N\bs G)$;
factorization homology $\int_{N\bs G}$ admits as input $N$-framed
little $T_{e}(N\bs G)$-disk algebras (see
Definition~\ref{defn:lilVdisk}).  The inner factorization homology
$\int_{K\bs N\times T_{e}(N\bs G)}X$ comes with a canonical such
structure.  In the case when $K$ acts trivially on $T_{e}G$, $N$
contains the identity component $G_{e}$, and $\int_{N\bs G}$ is a
finite smash power, so the two conjectures agree.

\section{Factorization homology of $H$-framed manifolds}\label{sec:fhne}

The construction of relative norms $N_H^G$ requires factorization homology for
non-parallelizable manifolds; we use this section to review framings
of smooth manifolds and how the framing fits into factorization
homology. We take a somewhat different approach from most other
treatments of factorization homology in that we work in terms of a
reduction of structure \textit{group} for the tangent bundle rather
than working with reduction of structure on the classifying space
level.  The work in this section is fundamentally non-equivariant,
though it intrinsically involves a structure group $H$ for the
framings.  We construct a point-set factorization homology functor in
this context, which we use in the next section to define relative
norms.

Before starting, we fix the following convention.  The first condition
ensures that the categories we consider are small.  The second
condition, while undesirable in some contexts, is convenient in the
context of factorization homology constructions.

\begin{conv}\label{conv:manifold}
When we use the term \term{manifold}, we will always understand that
its underlying topological space is a subspace of $\bR^{\infty}$ and
that it has finitely many components.
\end{conv}

We begin by fixing some terminology and notation.
For $M$ a smooth $d$-manifold, let $FM\to M$ denote the frame
bundle, a principal $\GL(d)$-bundle: a point consists of an
element $m$ of $M$ and a basis for the tangent space $T_{m}M$.
The frame bundle used $\bR^{d}$ as the standard model vector space,
but in our most important examples, $G$ and $G/H$, using the tangent
space at the identity $T_{e}G$ or at the identity coset $T_{eH}G/H$ is
more natural (choice-free) and so we formulate framings in terms of an 
arbitrary $d$-dimensional
vector space $V$.  A smooth $d$-manifold has a $V$-frame bundle $F_{V}M\to M$,
where a point in $F_{V}M$ consists of an element of $m$ of $M$ and a linear
isomorphism from $V$ to $T_{m}M$.  The $V$-frame bundle has the canonical
structure of a principal $\GL(V)$-bundle and there is a tautological
isomorphism of $\GL(V)$-bundles 
\[
F_{V}M\iso FM\times_{\GL(d)} \Iso(V,\bR^{d}),
\]
where $\Iso(V,\bR^{d})$ denotes the space of linear isomorphisms from
$V$ to $\bR^{d}$.  

\begin{ter}\label{ter:frame}
Let $H$ be a topological group with a given linear action on
$V$, i.e., a homomorphism $\rho \colon H\to \GL(V)$.  A
\term{tangential $H,V$-structure} on a smooth $d$-manifold 
$M$ (or \term{tangential $H$-structure}, when $V$, $\rho$ is
understood) consists of a principal $H$-bundle $P\to M$ and an
isomorphism of principal $\GL(V)$-bundles
\[
\phi \colon P\times_{H}\GL(V)\to F_{V}M.
\]
An \term{$H$-framed manifold} is a smooth $d$-manifold
together with a choice of 
tangential $H$-structure; we use the notation $F_{H}M\to M$ for its
structural principal $H$-bundle written $P\to M$ above and
$\phi_{M}$ for the structural isomorphism written $\phi$ above.
\end{ter}

When $H$ is the trivial group and $V=\bR^{d}$, we use
\term{parallelized} as a synonym for $H$-framed.

There is an obvious definition of maps of $H$-framed manifold in terms
of lifts of derivatives that we call $H$-framed local isometry, but it is
too constrained for many purposes.  The looser definition of
$H$-framed immersion is the right one for factorization homology.

\begin{defn}\label{defn:framemap}
Let $L$ and $M$ be $H$-framed manifolds (for fixed $\rho \colon H\to
\GL(V)$). An $H$-framed immersion $L\to M$ consists of a smooth
immersion (i.e., local diffeomorphism)
$f\colon L\to M$, a map of principal $H$-bundles
\[
Ff\colon F_{H}L\to f^{*}F_{H}M,
\]
and a principal $\GL(V)$-bundle homotopy
\[
If\colon F_{V}L\times I\to f^{*}F_{V}M
\]
such that $If$
ends at the derivative viewed as a map of frame bundles $F_{V}L\to
f^{*}F_{V}M$ and begins at the map 
\[
F_{V}L\overto{\phi_{L}^{-1}} F_{H}L\times_{H}\GL(V)\overto{Ff}f^{*}F_{H}M\times_{H}\GL(V)\overto{f^{*}\phi_{M}} f^{*}FM.
\]
induced by $Ff$.  An $H$-framed embedding is an $H$-framed immersion
whose underlying map is an open embedding (i.e., diffeomorphism onto
an open subset).  An $H$-framed diffeomorphism is a surjective
$H$-framed embedding, or equivalently, an $H$-framed embedding that
has an inverse, which is also an $H$-framed embedding.  An $H$-framed
local isometry is an $H$-framed immersion where the homotopy $If$ is
constant, i.e., the map of principal $H$-bundles $Ff$ is a lift of the
derivative.
\end{defn}

\begin{notn}
Let $\aE_{H}$ (or $\aE_{H,V}$ when specifying $V$ is needed for
clarity) denote the category whose objects are the $H$-framed 
manifolds and whose maps are the $H$-framed embeddings.  We topologize
the mapping spaces using the (k-ification of) the smooth compact open
topology on the space of smooth embeddings.
\end{notn}

When $V$ admits an $H$-invariant inner product, $H$-framed local
isometries are true local isometries for the resulting Riemmannian
structure; if not, the terminology fits less well.

\begin{example}
Let $V=\bR^{d}$, $H=O(d)$, with $\rho$ the standard inclusion. An
$H$-framing on a $d$-manifold $M$ then consists of a continuous (in
$m\in M$) choice of the orthogonal frames in $T_{m}M$, and so is
equivalent to the choice of a continuous Riemmannian metric on $M$.
Given $H$-framed manifolds $L$,$M$, an $H$-framed local isometry from
$L$ to $M$ consists of a smooth map $f\colon L\to M$ such that the
derivative at each point sends orthogonal frames to orthogonal frames,
i.e., it is precisely a local isometry in the usual sense.  An
$H$-framed immersion consists of a smooth immersion $f\colon L\to M$
and a $\GL(d)$-equivariant homotopy $If_{x}\colon FL_{x}\to
FM_{f(x)}$, continuous in $x\in L$, that starts at a map that takes
orthogonal frames to orthogonal frames and ends at $Df_{x}$.  Any
other such homotopy $I'f_{x}$ can be obtained by pointwise
multiplication by a path in $\GL(d)$ that starts at an element of
$O(d)$ and ends at the identity, that is to say, an element of the
homotopy fiber of $\rho$.  It follows that for fixed $f$, the space of
lifts of $f$ to an $H$-framed immersion is the space of sections of a
locally trivial principal $\hofib(\rho)$-bundle.  Since $\hofib(\rho)$
is contractible, so is this space of sections.
\end{example}

The work of the previous section implicitly used the following example.

\begin{example}\label{ex:G}
A Lie group $G$ has two canonical $H$-framings for $H$ the trivial
group and $V=T_{e}G$ the tangent space at the identity: one given by
left-invariant vector fields and the other given by right-invariant
vector fields. Choosing an isomorphism $\bR^{d}\iso T_{e}G$ gives $G$
a parallelization.  Consider the left-invariant framing or
parallelization.  For this framing, left multiplication by an element
of $G$ is an $H$-framed local isometry; moreover, it
is easy to see that when $G$ is connected, all $H$-framed local
isometries are of this form.  The inclusion of $G$ into the space of 
$H$-framed local isometries and into the space of $H$-framed
embeddings is continuous. 
\end{example}

We will always use the convention of left invariant vector fields. The
following generalization will form the basis for the relative norm.

\begin{example}\label{ex:GH}
Let $G$ be a Lie group and $H<G$ a closed subgroup. The orbit space
$G/H$ has a canonical $H,T_{eH}G/H$-tangential structure with
$H$-frame bundle the quotient map $G\to G/H$: given $g\in
G$, the derivative $DL_{g}|_{eH}$ of left multiplication by $g$ gives
an isomorphism of $T_{eH}(G/H)$ with the tangent space of $gH$.  Left
multiplication by elements of $G$ give $H$-framed local isometries.
\end{example}

Another important example is the vector space $V$ itself.

\begin{example}\label{ex:V}
The vector space $V$ viewed as an $d$-manifold has the canonical
structure of an $H$-framed manifold: the canonical identification of $V$ with
the tangent space $T_{v}V$ at every point $v\in V$ induces a
splitting of the $V$-frame bundle 
\[
F_{V}V\iso V\times \GL(V)
\]
and we use the split $H$-frame bundle 
\[
F_{H}V := V\times H
\]
with the map $F_{H}V\to F_{V}V$ induced by $\rho\colon  H\to \GL(V)$.  For an
$H$-framed manifold $M$ and an open embedding $f\colon V\to M$, a map
of $H$-principal bundles $F_{H}V\to f^{*}F_{H}M$ is determined by the
section
\[
V\iso V\times \{e\}\subset V\times H=F_{H}V\to f^{*}F_{H}M
\]
and an arbitrary section $V\to f^{*}F_{H}M$ induces a map of
$H$-principal bundles $F_{H}V\to f^{*}F_{H}M$.  Similar observations
apply to homotopies of maps of principal $\GL(V)$-bundles $F_{V}V\to
f^{*}F_{V}M$.  It follows that a lift of $f$ to an $H$-framed
embedding determines and is determined by a section $s$ of
$f^{*}F_{H}M$ and a homotopy over $V$ starting from the composite of
$s$ with the map $f^{*}F_{H}M\to F_{V}M$ induced by $\rho$ and ending
at the derivative, viewed as a map $V\to f^{*}F_{H}M$.  Similar
remarks apply to maps out of any $d$-manifold that has a given
splitting of its $V$-frame bundle.
\end{example}

\begin{conv}
For the rest of the section, we now fix the vector space $V$, the topological
group $H$, which we assume to be a Lie group, the homomorphism $\rho \colon H\to
\GL(V)$, and an $H$-invariant inner product structure on $V$, which we
assume exists.  (This in particular factors $\rho$ through $O(V)$, the
linear isometries of $V$.)
\end{conv}

In this context, factorization homology is built from the space of $H$-framed
embeddings of copies of $V$ in $H$-framed manifolds.  In the case
where the target is $V$ itself, the $H$-framed little disks operad
gives a small rigid model.

\begin{defn}\label{defn:lilVdisk}
Write $D$ for the open unit disk in $V$.  The
\term{$H$-framed little $V$-disk operad} $\oD^{V}_{H}$ has $n$th space
defined as follows: $\oD^{V}_{H}(0)$ consists of a single point.  An
element of $\oD^{V}_{H}(1)$ consists of a ordered pair $(\lambda,h)$
where $h\in H$ and $\lambda$ is a affine transformation $\lambda
\colon V\to V$ of the form
\[
\lambda(v)=v_{0}+rhv
\]
for some $v_{0}\in D$, $r\in (0,1]$ (and the given element $h\in H$)
that takes $D$ into $D$; 
$(\lambda,h)$ is then completely determined by $v_{0}$, $r$, $h$,
and we topologize $\oD^{V}_{H}(1)$ as a subspace of $D\times
(0,1]\times H$. For $n>1$, $\oD^{V}_{H}(n)$ is the subspace of
$\oD^{V}_{H}(1)\times \dotsb \times \oD^{V}_{H}(1)$ where the images of
$D$ under the affine transformations are disjoint.
The identity affine transformation and identity element of $H$ give an
identity element in $\oD^{V}_{H}(1)$.  We have a
right action of the symmetric group $\Sigma_{n}$ on $\oD(n)$ by
permuting the copies of $\oD^{V}_{H}(1)$, and we have an operadic
composition defined by composing maps and multiplying group elements:
the composition
\[
\oD^{V}_{H}(n)\times \oD^{V}_{H}(j_{1})\times \dotsb \times
\oD^{V}_{H}(j_{n})\to \oD^{V}_{H}(j)
\]
(for $j=j_{1}+\dotsb+j_{n}$)
is defined by
\[
\left(
\begin{gathered}
((\lambda_{1},h_{1}),\dotsc,(\lambda_{n},h_{n})),\\
((\mu_{1,1},g_{1,1}),\dotsc,(\mu_{1,j_{1}},g_{1,j_{1}})),\\
\dotsc,\\
((\mu_{n,1},g_{n,1}),\dotsc,(\mu_{n,j_{n}},g_{n,j_{n}}))
\end{gathered}\right)
\mapsto
\left(
\begin{gathered}
\shoveleft{(\lambda_{1}\circ \mu_{1,1},h_{1}g_{1,1}),\dotsc,}\\
\qquad\qquad (\lambda_{1}\circ \mu_{1,j_{1}},h_{1}g_{1,j_{1}}),
\\
\shoveleft{\quad\dotsc,}\\
\shoveleft{(\lambda_{n}\circ \mu_{n,1},h_{n}g_{n,1}),\dotsc,}\\
\qquad\qquad (\lambda_{n}\circ \mu_{n,j_{n}},h_{n}g_{n,j_{n}})
\end{gathered}\right).
\]
We emphasize that the operad $\oD^{V}_{H}$ depends on the action
$\rho$ of $H$ on $V$ and not just the abstract topological group $H$
and vector space $V$.
\end{defn}

\begin{rem}
When the kernel of $\rho \colon H \to O(V)$ is trivial, the element
$h$ of $H$ in the ordered pair in the definition of $\aD^{V}_{H}(1)$
(and $\aD^{V}_{H}(n)$) is completely determined by the affine
transformation $\lambda$ and we can more concisely define
$\oD^{V}_{H}(1)$ as the space of affine transformations
$\lambda \colon V\to V$ of the form $x\mapsto v_{0}+rhv$ that send the
unit disk into the unit disk ($\oD^{V}_{H}$ remains the subspace of
$\oD^{V}_{H}(1)^{n}$ where the images of the $D$ under the
affine transformations are disjoint).
\end{rem}

We contrast the $H$-framed little $V$-disk operad with the following
\term{$H$-equivariant little $V$-disk operad} often used in
equivariant factorization homology. 

\begin{defn}\label{defn:eqld}
The \term{$H$-equivariant little $V$-disk operad} $\oD_{V}$ has $n$th space
$\oD_{V}(n)$ the space of those ordered $n$-tuples of affine
transformations of the form
\[
v\mapsto v_{0}+rv
\]
that send the closed unit disk into the closed unit disk where (for
$n>1$) the images overlap only possibly on the boundary.  The identity
element is the identity map in $\oD_{V}(1)$ and the operadic
composition is induced by composing affine transformations.
The action of $H$ on $V$ induces an action of $H$ on the embedding
spaces by conjugation: $(h\cdot \lambda)(v)=h\lambda(h^{-1}v)$.  The
identity element $\id \in \oD_{V}(1)$ is fixed under this action, and
the operadic composition maps 
\[
\oD_{V}(n)\times \oD_{V}(j_{1}) \times \dotsb \times \oD_{V}(j_{n})\to \oD_{V}(j)
\]
(for $j=j_{1}+\dotsb+j_{n}$)
are $H$-equivariant, making $\oD_{V}$ an operad in the category of $H$-spaces.
\end{defn}

Non-equivariantly $\oD_{V}$ is $\oD^{V}_{e}$ where $e$ denotes the
trivial group, but $\oD_{V}$ and $\oD^{V}_{H}$ are related
equivariantly as follows.  Let $\oH$ be the operad with $\oH(n)=H^{n}$, where
composition is induced by diagonal maps and the group multiplication:
\[
\left(
\begin{gathered}
(h_{1},\dotsc,h_{n}),\\
(k_{1,1},\dotsc,k_{1,j_{1}}),\\
\dotsc,\\
(k_{n,1},\dotsc,k_{n,j_{n}}
\end{gathered}\right)
\mapsto
\left(\begin{gathered}
h_{1}k_{1,1},\dotsc,h_{1}k_{1,j_{1}}\\
\dotsc,\\
h_{n}k_{1,j_{n}},\dotsc,h_{n}k_{n,j_{n}}
\end{gathered}
\right).
\]
The $H$-framed little $V$-disks operad
$\oD^{V}_{H}$ is then isomorphic to the semidirect product
$\oD_{V}\rtimes \oH$ of the $H$-equivariant little $V$-disks operad
$\oD_{V}$ and the operad $\oH(n)=H^{n}$, with composition
on the $\oD_{V}$ factor twisted by the action of $H$ on $\oD(j_{i})$:
the isomorphism $\oD_{V}\rtimes \oH\to \oD^{V}_{H}$ is given on the
$n$-ary part by the formula
\begin{multline}\label{eq:FtoE}
((\lambda_{1},\dotsc,\lambda_{n}),(h_{1},\dotsc,h_{n}))
\in \oD_{V}(n)\times H^{n}=(\oD_{V}\rtimes \oH)(n)\\
\mapsto
((\lambda_{1}\circ h_{1},h_{1}),\dotsc,(\lambda_{n}\circ
h_{n},h_{n}))\in \oD^{V}_{H}(n).
\end{multline}
The relationship between the operads is given by the following proposition, which is
purely formal and holds in any suitable topological category (see, for
example, \cite[2.3]{SalvatoreWahl}). In it, we understand an
$H$-equivariant $\oD_{V}$-algebra to be an object $X$ with an action
of both the operad $\oD_{V}$ and the topological group $H$ such that
the algebra structure maps are $H$-equivariant.

\begin{prop}\label{prop:ldalg}
The point-set category of $H$-equivariant $\oD_{V}$-algebras is isomorphic to
the point-set category of $\oD^{V}_{H}$-algebras.
\end{prop}

The difference between the two categories is then purely structural or
philosophical.  We use $\oD^{V}_{H}$ when we need to work in a
non-equivariant context and $\oD_{V}$ when we need to work in an
equivariant context.  In terms of homotopy categories or $\infty$-categories, viewing
$\oD^{V}_{H}$ merely as an operad, the natural notion of weak
equivalence on $\oD^{V}_{H}$-algebras would correspond to Borel
equivalence of $H$-equivariant $\oD_{V}$-algebras, which is never what
we want here.  As a category of $H$-equivariant orthogonal spectra
with extra structure, the category of $H$-equivariant
$\oD_{V}$-algebras, we have $U$-universe homotopy theories for all
$H$-universes. 

As indicated above its definition, the $H$-framed little $V$-disk operad
models $H$-framed embeddings of copies of $V$ into $V$.  To make this
precise, we first note that the open disk $D$ is $H$-framed
diffeomorphic to $V$ where we used the canonical splitting of the
$V$-frame bundle of $D$ to define the $H$-framed structure.  Choosing
and fixing an $H$-framed diffeomorphism, it suffices to discuss $H$-framed
embeddings of copies of $D$.  We have a continuous map
\[
\oD^{V}_{H}(n)\to \aE_{H}(D\times \{1,\dotsc,n\},D)
\]
defined as follows.  We can specify the $H$-framed structure on a
smooth map as in Example~\ref{ex:V}. For an element $(\lambda,h)$ of
$\oD^{V}_{H}(1)$, the underlying smooth map $D\to D$ is $\lambda$, the
element $h$, viewed as a constant section
\[
D\to f^{*}F_{H}D\iso D\times H
\]
induces the lift of frame bundles, and we use $t\mapsto r^{t}\rho (h)$
as the homotopy over $D$
\[
D\times I\to f^{*}F_{V}D\iso D\times \GL(V)
\]
from the image of the lift to the derivative (where $r$ is as in the
notation of Definition~\ref{defn:lilVdisk}, $\lambda(v)=v_{0}+rhv$).
This specifies a continuous map $\oD^{V}_{H}(1)\to \aE_{H}(D\times
\{1\},D)$, and for $n>1$, the element
$((\lambda_{1},h_{1}),\dotsc,(\lambda_{n},h_{n}))$ goes to the
$H$-framed map $D\times \{1,\dotsc,n\}\to D$ that does the lift of
$(\lambda_{i},h_{i})$ just described on the $i$th copy.
Taken together, the collection $\aE_{H}(D\times \{1,\dotsc,n\},D)$
(as $n$ varies) forms an operad, a version of the endomorphism operad
$\oEnd_{\amalg}(D)$ in $\aE_{H}$ (for the symmetric monoidal product
given by disjoint union). The following observation is clear by
construction. 

\begin{prop}
The map $\oD^{V}_{H}(n)\to \aE_{H}(D\times \{1,\dotsc,n\},D)$ is
compatible with the symmetric group action and composition, giving a
map of operads $\oD^{V}_{H}\to \oEnd_{\amalg}(D)$.
\end{prop}

It well-known and well-documented in the literature that the map
$\oD^{V}_{H}(n)\to \aE_{H}(D\times \{1,\dotsc,n\},D)$ is a homotopy
equivalence.  In fact, we can say more than this.  There is an obvious
inclusion of the wreath product group $\Sigma_{n}\wr H$ in the
$H$-framed self-diffeomorphisms of $D\times \{1,\dotsc,n\}$, where
$\Sigma_{n}$ permutes the factors and elements of $H$ act by the
$H$-isometries on each summand (an element $h$ of $H$ has underlying
smooth map $\rho(h)$, lift specified by the constant section $h$, and
the homotopy over $D$ also constant).  This induces a natural action of
$\Sigma_{n}\wr H$ on $\aE_{H}(D\times \{1,\dotsc,n\},M)$ (for any
$H$-framed manifold $M$), and for $M=D$, it
restricts to a compatible action on $\oD^{V}_{H}(n)$.  In [CFH], we
prove the following theorem about this $\Sigma_{n}\wr H$-action.

\begin{thm}\label{thm:DEcomp}\ 
\begin{enumerate}
\item  For any $H$-framed manifold $M$, the $\Sigma_{n}\wr H$-space
$\aE_{H}(D\times \{1,\dotsc,n\},M)$ is equivariantly homotopy
equivalent to a free $\Sigma_{n}\wr H$-CW complex.
\item The map $\oD^{V}_{H}(n)\to \aE_{H}(D\times \{1,\dotsc,n\},D)$ is a 
$\Sigma_{n}\wr H$-equivariant homotopy equivalence. 
\end{enumerate}
\end{thm}

We construct factorization homology as a homotopy coend for the (left) action
of the operad $\oD^{V}_{H}$ on an spectrum and the right action of
$\oD^{V}_{H}$ on the following embedding spaces.

\begin{notn}
For an $H$-framed manifold $M$, let $\oE_{M}(n)=\aE_{H}(D\times
\{1,\dotsc,n\},M)$. 
\end{notn}

We have a map
\[
\oE_{M}(n)\times (\oD^{V}_{H}(j_{1},1)\times \dotsb \times \oD^{V}_{H}(j_{n},1))\to \oE_{M}(j)
\]
(for $j=j_{1}+\dotsb+j_{n}$) obtained by composing $H$-framed
embeddings. The collection $\oE_{M}(n)$ forms a symmetric sequence
(which is just to say that each $\oE_{M}(n)$ comes with an action of
$\Sigma_{n}$), and the map above can be re-interpreted in the
category of symmetric sequences as a right action of $\oD^{V}_{H}$ on
$\oE_{M}$ for the plethysm product.  This is simpler to explain in
terms of associated functors: consider the endofunctors $\bE_{M}$ and
$\bD$ on orthogonal spectra defined by
\[
\bE_{M}X=\bigvee_{n\geq 0}\oE_{M}(n)_{+}\sma_{\Sigma_{n}}X^{(n)},\qquad 
\bD X=\bigvee_{n\geq 0}\oD^{V}_{H}(n)_{+}\sma_{\Sigma_{n}}X^{(n)}.
\]
Then $\bD$ is the monad associated to the operad $\oD^{V}_{H}$, and 
the composition maps above define a right action of $\bD$ on $\bE_{M}$
\[
\bE_{M}\circ \bD\to \bE_{M}
\]
in the category of endofunctors of orthogonal spectra (in terms of
composition).  In this setting, we have the monadic bar construction
of May~\cite[\S 9]{GILS}:

\begin{cons}\label{cons:monadic}
Let $M$ be an $H$-framed manifold and let $A$ be a
$\oD^{V}_{H}$-algebra in the category of orthogonal spectra.  Define
the simplicial object $B\subdot(M;A)$ to be the monadic bar
construction $B\subdot(\bE_{M},\bD,A)$:
\[
B_{q}(M;A):= B_{q}(\bE_{M},\bD,A)
  =\bE_{M}\underbrace{\bD\dotsb\bD}_{n\text{ factors}} A
\]
where the face maps $d_{i}$ are induced by the monadic composition $\bD\bD\to
\bD$ (for $0<i<q$), the action of $\oD^{V}_{H}$ on $A$, $\bD A\to A$
(for $i=0$), and the right action of $\oD^{V}_{H}$ on $\oE_{M}$,
$\bE_{M}\bD\to \bE_{M}$ (for $i=q$).  The degeneracy maps $s_{i}$ are
induced by the monadic unit maps $\Id\to \bD$.  We write $B(M;A)$ for the
geometric realization, or $B_{H}(M;A)$ when it is necessary to
denote or emphasize the structure group $H$
\end{cons}

The previous construction $B(M;A)$ is a standard formulation of
factorization homology $\int_{M}A$ in the context of $H$-framed
manifolds, at least under some cofibrancy hypotheses on $A$; see, for
example, \cite[IX.1.5]{Andrade-Thesis}, \cite[5.5.2.6]{Lurie-HA},
\cite[\S2.3]{Miller-Nonabelian}, \cite[Def.~35]{KupersMiller-EnCells},
\cite[3.14]{Zou-EqFactHom}.  Indeed,
because $\oE_{M}(n)$ and $\oD^{V}_{H}(n)$ are
$\Sigma_{n}$-equivariantly homotopic to free $\Sigma_{n}$CW complexes,
$B(M;-)$ preserves weak equivalences in $\oD^{V}_{H}$-algebras $A$
whose underlying orthogonal spectra are ``flat'' in the sense that
the point-set smash product functor $A\sma(-)$ preserves weak
equivalences.  It even suffices for the underlying orthogonal spectra
just to have the right smash powers in the sense that the map in the
stable category from the derived smash power to the point-set smash
power is a weak equivalence. $B(M;A)$ correctly computes $\int_{M}A$
just under this kind of minimal hypothesis on $A$.

When a topological group $G$ acts on $M$ through $H$-framed
diffeomorphisms, $B(M;A)$ obtains a natural $G$-action.  When $H$ is
the trivial group, this gives the point-set construction of
factorization homology from [CFH] with the properties we
asserted in the previous section.  When $H$ is non-trivial, the
construction is too flabby to have the correct $G$-equivariant
homotopy type.  We correct this with the following ``compressed'' bar
construction. 

\begin{cons}\label{cons:compressed}
Let $\bar \bD$ be the free  $H$-equivariant $\oD_{V}$-algebra monad on
$H$-equivariant orthogonal spectra
\[
\bar \bD X := \bigvee_{n\geq 0}\oD_{V}(n)\times_{\Sigma_{n}}X^{(n)}
\]
and for an $H$-framed manifold $M$, let $\bar
E_{M}$ denote the functor from $H$-equivariant orthogonal spectra to
orthogonal spectra defined by
\[
\bar \bE_{M} X := \bigvee_{n\geq 0}\oE_{M}(n)\times_{\Sigma_{n}\wr H}X^{(n)}.
\]
For a $\oD^{V}_{H}$-algebra $A$ (viewed as an $H$-equivariant
$\oD_{V}$-algebra),  define 
the simplicial object $\bar B\subdot(M;A)$ to be the monadic bar
construction $B\subdot(\bar\bE_{M},\bar\bD,A)$:
\[
\bar B_{q}(M;A):= B_{q}(\bar\bE_{M},\bar\bD,A)
  =\bar\bE_{M}\underbrace{\bar\bD\dotsb\bar\bD}_{n\text{ factors}} A
\]
with the usual face and degeneracy maps (see
Construction~\ref{cons:monadic}). Write $\bar B(M;A)$ for the
geometric realization, or $\bar B_{H}(M;A)$ when it is necessary to
denote or emphasize the structure group $H$.
\end{cons}

In terms of the $H$-framed little $V$-disk operad, the monad $\bar
\bD$ is naturally isomorphic to the monad (on $H$-equivariant
orthogonal spectra)
\[
\bigvee_{n\geq 0}\oD^{V}_{H}(n)_{+}\sma_{\Sigma_{n}\wr H}X^{(n)}
\]
where the resulting $H$-action is the $H$-action from the
$\oD^{V}_{H}$-algebra structure.  In concrete terms, the $H$-action on
the $n$th summand is 
induced by the left $H$-action on $\oD^{V}_{H}(n)$ coming from the
$H$-action on $D$ in the category of $H$-framed manifolds.
(Specifically, $h\in H$ sends
$((\lambda_{1},h_{1}),\dotsc,(\lambda_{n},h_{n}))$ to $((h\circ
\lambda_{1},hh_{1}),\dotsc,(h\circ \lambda_{n},hh_{n}))$.)
Under the isomorphism $\oD^{V}_{H}(n)\iso \oD^{V}(n)\times H^{n}$
of~\eqref{eq:FtoE}, this action corresponds to the diagonal action on 
\[
\oD_{V}(n)_{+}\sma_{\Sigma_{n}}X^{(n)}
\]
used in Construction~\ref{cons:compressed}.  This makes clear the
relationship between $B(M;A)$ and $\bar B(M;A)$: expanding out the
definitions in terms of the spaces $\oE_{M}(n)$ and $\oD^{V}_{H}(n)$,
$B(M;A)$ is formed from products of these smashed with smash powers of
$A$ by coequalizing symmetric group actions and $\bar B(M;A)$ is
formed by the same products and smash powers by coequalizing the
action of the corresponding wreath product with $H$.

The quotient map
\[
B(M;A)\to \bar B(M;A)
\]
is natural in both the $H$-framed manifold $M$ and the
$\oD^{V}_{H}$-algebra $A$, and a straight-forward ``Quillen
Theorem~A'' argument (plus Theorem~\ref{thm:DEcomp}.(i)) proves that
it is always a homotopy equivalence.  In particular, this is a natural
weak equivalence, but the homotopy inverse and homotopy data cannot be
made natural in $M$.

\section{The relative case}\label{sec:rel}

We now consider the case when $H<G$ is a closed subgroup of a positive
dimensional compact Lie group $G$ and discuss the relative norm
$N_{H}^{G}$.  In the case when $H<G$ is finite index, we already know
how to construct this norm as an equivariant smash
power~\cite[8.1]{HillHopkins-EqSymMon}; the more interesting case is
when $H<G$ is positive codimension.  In the finite index case, the
relative norm makes sense for any genuine $H$-spectrum; in the
positive codimension case, we need additional structure of precisely
the kind introduced in the previous section.   

Let $G$ be a compact Lie group and $H<G$ a closed subgroup.  Let
$V=T_{eH}G/H$ with its natural $H$-action. We then have a
canonical tangential $H,V$ structure (q.v. Terminology~\ref{ter:frame}
and Example~\ref{ex:GH}) on $G/H$ and the left multiplication action
of $G$ on $G/H$ is an action in the
category of $H$-framed manifolds.  Our setup in the
previous section assumed an $H$-invariant inner product on $V$; as the
space of such inner products is contractible and our constructions are
continuous, we can choose one arbitrarily (but a uniform way to choose
the inner product for all $H<G$ at once is to choose a $G$-invariant
inner product on $T_{e}G$).  As in general for norms, we work with
complete universes: we fix a complete $G$ universe $U$, which we
assume (wlog) contains $\bR^{\infty}$ and also a complete $H$-universe
$U_{H}$, which contains $\bR^{\infty}$.  (We can as in
Section~\ref{sec:absolute} also assume that $U^{G}=\bR^{\infty}$ and
$(U_{H})^{H}=\bR^{\infty}$ if we so choose, but this is not necessary,
and without this requirement, a uniform way to choose $U_{H}$ for all
$H<G$ at once is to use $U$ with action restricted to $H$.)

\begin{defn}\label{defn:relnorm}
For $X$ an $H$-equivariant $\aD_{V}$-algebra indexed on $U_{H}$, we
define the relative norm $N_{H}^{G}X$ as the point-set functor
\[
N_{H}^{G}X=I_{\bR^{\infty}}^{U}\bar B(G/H;I^{\bR^{\infty}}_{U_{H}}X)
\]
using the construction $\bar B$ of~\ref{cons:compressed} and point set
change of universe functors $I$.  Here the $G$-action on $\bar
B(G/H;-)$ comes from topological functoriality of $\bar B$ and the
action of $G$ on $G/H$ in the category of $H$-framed manifolds.
\end{defn}

Just as in the absolute case, the relative norm comes with a
cardinality filtration from the Ayala-Francis cardinality filtration
on (non-unital) factorization homology.  In this case, the associated graded
cofiber at filtration level $q$ looks (non-equivariantly) like 
\[
C^{\epi}_{H}(q,G/H)^{+}\sma_{\Sigma_{q}\wr H}X^{(q)}
\]
where $C_{H}(q,M)$ denotes the $H$-framed version of the configuration
space: an element consists of a $q$-tuple of elements of $F_{H}M$
whose image in $M$ is a configuration (and $C_{H}^{\epi}$ is
the subspace where the configuration in $M$ is surjective on $\pi_{0}$).
Equivariantly, we expect this piece to be weakly equivalent to
\[
I_{\bR^{\infty}}^{U}(C^{\epi}_{H}(q,G/H)^{+}\sma_{\Sigma_{q}\wr H}(I^{\bR^{\infty}}_{U_{H}}X)^{(q)}
\]
at least under the hypothesis that the underlying $H$-equivariant
orthogonal spectrum of $X$ is cofibrant in the $U_{H}$-universe model
structure or $X$ is an $H$-equivariant commutative ring orthogonal
spectrum and cofibrant in the commutative ring $U_{H}$-universe model
structure.  If this is the case, expected results from [PMI] would
then imply that the following conjecture holds.

\begin{conj}\label{conj:rhi}
Let $X'$ and $X$ satisfy the Cofibrancy Hypothesis~\ref{hyp} below.
For the point-set construction of
Definition~\ref{defn:relnorm}, a $U_{H}$-universe weak equivalence
$X'\to X$ induces a $U$-universe weak equivalence 
$N_{H}^{G}X'\to N_{H}^{G}X$.
\end{conj}

\begin{hyp}\label{hyp}
For the purposes of this section, we say that an $H$-equivariant $\oD_{V}$-algebra
satisfies the ``Cofibrancy Hypothesis'' if one of the following holds:
\begin{itemize}
\item Its underlying
$H$-equivariant orthogonal spectrum is cofibrant in the
$U_{H}$-universe model structure 
\item It inherits its $H$-equivariant $\oD_{V}$-algebra by virtue of
being an $H$-equivariant commutative
ring orthogonal spectrum, and it is cofibrant in the $U_{H}$-universe
model structure on commutative ring orthogonal spectra.  
\end{itemize}
\end{hyp}

Restricting to the second case in the hypothesis, the following is an
immediate corollary of the conjecture.

\begin{cor}[Conjectural]\label{cor:rhi}
The left derived functor of $N_{H}^{G}$ exists.  Moreover, the
composite with the derived functor (for the $U_{H}$-universe homotopy
categories) of the forgetful functor $\iota$ from $H$-equivariant
commutative ring orthogonal spectra to $H$-equivariant
$\oD_{V}$-algebras is the derived functor of the composite
$N_{H}^{G}\iota$.
\end{cor}

When $X$ is an $H$-equivariant commutative ring orthogonal spectrum,
we would like to compare $N_{H}^{G}X$ to the left adjoint functor of
the forgetful functor from $G$-equivariant commutative ring orthogonal
spectra to $H$-equivariant commutative ring orthogonal spectra.
Denote this left adjoint as $(-)\otimes_{H}G$.  The point-set model,
up to isomorphism, does not depend on the indexing universe, and
indexing on $\bR^{\infty}$, we can identify
$(I^{\bR^{\infty}}_{U_{H}}X)\otimes_{H}G$ as a quotient of the free
$G$-equivariant 
commutative ring orthogonal spectrum on
$(I^{\bR^{\infty}}_{U_{H}}X)\sma_{H}G_{+}$. Filtering this with the 
$q$th level the image of
$((I^{\bR^{\infty}}_{U_{H}}X)\sma_{H}G_{+})^{(q)}/\Sigma_{q}$, the 
associated graded point-set quotients are then given by the orthogonal
$G$-spectra 
\[
C_{H}(q,G/H)_{+}\sma_{\Sigma_{q}\wr H} (I^{\bR^{\infty}}_{U_{H}}(X/\bS))^{(q)}
\]
which we can re-index to $U$ using the point-set change of universe
$I^{U}_{\bR^{\infty}}$.  When the inclusion of $\bS$ in $X$ is a
Hurewicz cofibration of orthogonal $H$-spectra, the maps in the
filtration are Hurewicz cofibrations of orthogonal $G$-spectra, and we
can use this filtration to analyze $X\otimes_{H}G$ homotopically.
Just as in the absolute case, factorization homology of unital
algebras has a unital version of the cardinality filtration, with the
$q$-level associated graded cofiber expected (under suitable
cofibrancy hypotheses) to be $U$-universe weakly equivalent to
\[
I_{\bR^{\infty}}^{U}(C_{H}(q,G/H)^{+}
\sma_{\Sigma_{q}\wr H}(I^{\bR^{\infty}}_{U_{H}}(X/\bS))^{(q)}),
\]
that is, the $U$-re-indexed associated graded quotient above. If this
all works, it would then establish the following conjecture.

\begin{conj}
Let $R$ be a cofibrant $H$-equivariant commutative ring orthogonal
spectrum in the $U_{H}$-universe model structure.  Then there is a
natural $U$-universe weak equivalence 
\[
N_H^G R \htp R \otimes_{H} G
\]
where $(-) \otimes_{H} G$ is the left adjoint of the
(point-set) forgetful functor from $G$-equivariant commutative ring
orthogonal spectra to $H$-equivariant commutative ring orthogonal
spectra. 
\end{conj}

Analyzing the cardinality filtration gives conjectures for the
geometric fixed points:

\begin{conj}
Let $X$ be an $H$-equivariant $\oD_{V}$-algebra satisfying the
Cofibrancy Hypothesis~\ref{hyp} above.  If $K$ is a normal
subgroup of $G$ and $H$ is finite index in $HK$, then $H\cap K$ acts
trivially on $V$, the map of $H$-equivariant inner product spaces $V=T_{eH}G/H\to
T_{eHK}G/HK$ is an isomorphism, and there exists a natural 
$U^{K}$-universe $G/K$-equivariant weak equivalence
\[
(N_{H}^{G}X)^{\Phi K}
\simeq N_{(G/K)/(HK/K)}^{G/K}(X^{\Phi (H\cap K)}).
\]
If $K$ is a normal subgroup of $G$ and $H$ is not finite index in
$HK$, then the unit map induces a $U^{K}$-universe $G/K$-equivariant
weak equivalence $\bS\to (N_{H}^{G}X)^{\Phi K}$.
\end{conj}

Having formulated a relative construction, it is now possible to
try to iterate norms.  For $K<H<G$, we expect an equivalence between
$N_{K}^{G}X$ and an iterated construction along the lines of
$N_{H}^{G}N_{K}^{H}X$.  The first issue that arises is that the inputs
for $N_{K}^{G}$ and $N_{K}^{H}$ do not match: the former wants a
$K$-equivariant little $T_{eK}(G/K)$-disk algebra whereas the latter wants
a $K$-equivariant little $T_{eK}(H/K)$-disk algebra.  This is minor because
(having chosen a $G$-invariant inner product on $T_{e}G$), we have a
decomposition of $K$-equivariant inner product spaces
\begin{equation}\label{eq:Tdecomp}
T_{eK}(G/K) \iso T_{eK}(H/K)\oplus T_{eH}(G/H),
\end{equation}
and this gives a forgetful functor from the input for
$N_{K}^{G}$ to the input for $N_{K}^{H}$ at least on the level of
homotopy categories.  A more serious issue is
that even when we use a $K$-equivariant little $T_{eK}(G/K)$-disk
algebra $X$ as the input, the output of $N_{K}^{H}$ does not obviously
have the structure of a $H$-equivariant little $T_{eH}G/H$-disk
algebra, which is what is needed as the input to $N_{H}^{G}$.  On the
other hand, when $X$ is a $K$-equivariant little $T_{eK}(G/K)$-disk
algebra, we can make sense of the factorization homology
\[
\bar B(H/K \times D(T_{eH}(G/H));X)
\]
where $D$ denotes the open unit disk and we understand $H/K \times
D(T_{eH}(G/H))$ as a $K,T_{eK}(G/K)$-framed manifold using the
isomorphism~\eqref{eq:Tdecomp} again. We expect the following to hold:

\begin{conj}
Let $K<H<G$ and let $X$ be a $K$-equivariant little $T_{eK}(G/K)$-disk
algebra.  There is a natural zigzag of $H$-equivariant homotopy equivalence
\[
\bar B(H/K \times D(T_{eH}(G/H));X)
\simeq B(H/K,i^{*}X)
\]
where $i^{*}$ denotes a functor to $K$-equivariant little
$T_{eK}(G/K)$-disk algebras modeling the homotopical forgetful functor
for the decomposition in~\eqref{eq:Tdecomp}.
\end{conj}

In particular, $I_{\bR^{\infty}}^{U_{H}}\bar B(H/K \times
D(T_{eH}(G/H));I_{U_{K}}^{\bR^{\infty}}X)$ should be an acceptable
stand-in for $N_{K}^{H}(i^{*}X)$.

The advantage of $\bar B(H/K \times D(T_{eH}(G/H));X)$ is that the
$K$-framed manifold $H/K \times D(T_{eH}(G/H))$ comes with the
structure of a $\oD^{T_{eH}(G/H)}_{H}$-algebra in the category
$\aE_{K}$, as we now explain.  The key observation is that the
(diagonal) (left) multiplication by $H$ on $H/K\times D(T_{eH}G/H)$ is a
$K$-framed local isometry.  The tangential $K,T_{eK}(G/K)$-structure on
$H/K\times D(T_{eH}G/H)$ has $K$-frame bundle $H\times D(T_{eH}G/H)$
with $K$ acting on the right of $H$ (acting trivially on $D(T_{eH}G/H)$).
The identification of 
\[
(H\times D(T_{eH}G/H))\times_{K} T_{eK}(G/K)
\]
with the tangent bundle of $H/K\times D(T_{eH}G/H)$ sends an element 
\begin{multline*}
((h,u),v\oplus w)\in (H\times D(T_{eH}G/H))\times_{K} (T_{eK}(H/K)\oplus T_{eH}(G/H)) \\
\iso (H\times D(T_{eH}G/H))\times_{K} T_{eK}(G/K)
\end{multline*}
to $((h,u), DL_{h}|_{eK}v\oplus h\cdot w)$, where
$DL_{h}|_{eK}$ denotes the derivative of left multiplication by $h$ on
$H/K$.  For an element $g\in H$, multiplication by $g$ on $H/K\times
D(T_{eH}G/H)$ then clearly sends $K$-frames to $K$-frames.  Given an
element 
\[
((\lambda,h_{1}),\dotsc,(\lambda_{n},h_{n}))\in \oD^{T_{eH}(G/H)}_{H}(n)
\]
of the
$H$-framed little $T_{eK}(G/H)$-disk operad, we then get a $K$-framed
embedding
\[
(H/K\times D(T_{eH}G/H))\amalg \dotsb \amalg (H/K\times
D(T_{eH}G/H))\to (H/K\times D(T_{eH}G/H)) 
\]
using the group elements $h_{i}$ diagonally and the affine
transformations $\lambda$ on the disks $D(T_{eH}G/H)$.

A fundamental property of factorization homology is that it is
symmetric monoidal in both variables.  In particular, 
our construction $\bar B$ is strong symmetric monoidal on the
point-set level in the manifold variable: it takes disjoint unions of
manifolds to smash products of orthogonal spectra up to coherent
natural isomorphism.  It follows that an operadic algebra structure on
the manifold $M$ (for disjoint union in the category $\aE_{K}$)
induces the same kind of operadic algebra structure on $\bar B(M;X)$
for fixed $X$.  In the present context, this discussion proves the
following observation. 

\begin{prop}
Let $K<H<G$ and let $X$ be a $K$-equivariant little $T_{eK}(G/K)$-disk
algebra.  The orthogonal spectrum $\bar B(H/K \times
D(T_{eH}(G/H));X)$ has the canonical structure of a
$\oD^{T_{eH}(G/H)}_{H}$-algebra. 
\end{prop}

This allows us to make sense of the iterated norm construction.  We
conjecture the following relationship.

\begin{conj}
Let $K<H<G$ and let $X$ be a $K$-equivariant little $T_{eK}(G/K)$-disk
algebra whose underlying $K$-equivariant orthogonal spectrum is
cofibrant in the $U_{K}$-universe.  There is a natural 
$U$-universe $G$-equivariant weak equivalence
\[
N_{K}^{G}X \simeq 
N_{H}^{G} I_{\bR^{\infty}}^{U_{H}}
\bar B(H/K \times D(T_{eH}(G/H));I^{\bR^{\infty}}_{U_{K}}X).
\]
\end{conj}

As mentioned above, factorization homology is symmetric monoidal in
both variables.  Symmetric monoidality in the algebra variable should
imply this expected property of norms.

\begin{conj}
Let $X$ and $Y$ be $H$-equivariant $\oD_{V}$-algebras whose 
underlying $H$-equivariant orthogonal spectra are cofibrant in the
$U_{H}$-universe model structure. There is a natural 
$U$-universe $G$-equivariant weak equivalence
\[
N_{H}^{G}X \sma N_{H}^{G}Y \simeq N_{H}^{G}(X\sma Y).
\]
\end{conj}

\section{The case of the circle group}\label{sec:circle}

The paper~\cite{ABGHLM} defines norms and relative norms for the
circle group $S^{1}$ in terms of cyclic bar constructions.  In this
section, we compare the point-set construction of the norms
in~\cite[1.1,8.2]{ABGHLM} to the point-set construction of the norms
in the previous section in the case $G=S^{1}$; see
Theorem~\ref{thm:circle} for a precise statement.  This section is
independent of the work of [CFH] and [PMI], and depends on the rest of
this paper only in its use of the notation, terminology, and
definitions.

We fix the positive integer $m$ and consider the subgroup
$C_{m}<S^{1}$.  To avoid notational confusion in what follows, we
consistently write $Z_{n}$ for the cyclic subgroup of $\Sigma_{n}$
generated by the cyclic permutation $(1\dotsb n)$.

Let $U$ denote a complete $S^{1}$-universe; we write
$U_{C_{m}}$ for $U$ regarded as a complete $C_{m}$-universe.  The
point-set \cite[8.2]{ABGHLM} norm is a functor from
$C_{m}$-equivariant associative ring spectra indexed on $U_{C_{m}}$ to
$S^{1}$-equivariant orthogonal spectra indexed on $U$ built as a
composite
\[
I_{\bR^{\infty}}^{U}N^{cyc,C_{m}}_{\sma}I_{U_{C^{n}}}^{\bR^{\infty}}
\]
using a $C_{m}$ relative cyclic bar construction (which we review
starting after Theorem~\ref{thm:cycbar} below) and point-set change of universe
functors.  For the point-set norm of Definition~\ref{defn:relnorm}, we 
note that the action of $C_{m}$ on the tangent space
$T_{eC_{m}}S^{1}/C_{m}$ is trivial, and (to simplify notation) we
identify this tangent space with $\bR$ using the standard metric and
orientation on $S^{1}$ (with total length $2\pi$, giving $S^{1}/C_{m}$
length $2\pi/m$).  The norm 
\[
I_{\bR^{\infty}}^{U}\bar B_{C_{m}}(S^{1}/C_{},I_{U_{C^{n}}}^{\bR^{\infty}}(-))
\]
is then a functor from $C_{m}$-equivariant $\oD_{\bR}$-algebras
(little $1$-disk algebras) in orthogonal spectra indexed on
$U_{C_{m}}$ to $S^{1}$-equivariant orthogonal spectra indexed on $U$.
The inputs for these functors differ, but we have a forgetful functor
from $C_{m}$-equivariant associative ring orthogonal spectra to
$C_{m}$-equivariant $\oD_{\bR}$-algebras (which is an equivalence on
homotopy categories), so we state the comparison
in the associative case.  The following is the main result of this
section. 

\begin{thm}\label{thm:circle}
Let $R$ be a $C_{m}$-equivariant associative ring orthogonal
spectrum indexed on $\bR^{\infty}$. There is a natural zigzag of
natural $S^{1}$-equivariant homotopy equivalences 
\[
\bar B_{C_{m}}(S^{1}/C_{m},R)\from
B(\bar {\mathrm{E}},\bar \bD,R)\to
B(\bar {\mathrm{E}}^{c},\bar \bD^{c},R)\from 
B(\bar {\mathrm{C}}^{c},\bT,R)
\]
and a natural isomorphism 
\[
B(\bar{\mathrm C}^{c},\bT,R)\iso N^{cyc,C_{m}}_{\sma}B(\bT,\bT,R)
\]
of $S^{1}$-equivariant orthogonal spectra indexed on $\bR^{\infty}$.
\end{thm}

The functors of the form $B(-,-,R)$ are all monadic bar constructions,
which we explain in detal below.  In the first display, all the maps
are geometric realizations of maps that are natural
$S^{1}$-equivariant homotopy equivalences on each simplicial level
(the homotopy inverses and homotopy data is natural in $R$). See
Propositions~\ref{prop:zig1}, \ref{prop:zig2}, \ref{prop:zig3} below.
In both displays, $\bT$ denotes the free associative algebra monad (in
the category of $C_{m}$-equivariant orthogonal spectra) and in the
second display, $B(\bT,\bT,R)$ is the geometric realization of the
``standard construction'' or the two-sided bar construction.  The
$C_{m}$-equivariant associative ring spectrum $B(\bT,\bT,R)$ comes
with a natural map in the category of $C_{m}$-equivariant associative
ring spectra
\[
B(\bT,\bT,R)\to R
\]
which is a natural $C_{m}$-equivariant homotopy equivalence of the
underlying orthogonal spectra.  Because point-set change of universe
functors are topologically enriched and therefore preserve homotopy
equivalences, we get the following norm comparison as a corollary
(applying the previous theorem to $I_{C_{m}}^{\bR^{\infty}}A$ and
applying $I_{\bR^{\infty}}^{U}$ to the resulting zigzag).

\begin{cor}\label{cor:comparenorms}
Let $A$ be a $C_{m}$-equivariant associative ring orthogonal
spectrum indexed on $U_{C_{m}}$.  The relative norm
$N_{C_{m}}^{S^{1}}A$ of Definition~\ref{defn:relnorm} is naturally
$S^{1}$-equivariantly homotopy equivalent to the relative norm
$N_{C_{m}}^{S^{1}}B(\bT,\bT,A)$ of \cite[8.2]{ABGHLM}.
\end{cor}

The relative norm of~\cite[8.2]{ABGHLM} is known to preserve the weak
equivalences used there (the ``$\aF$-equivalences'') under mild
hypotheses on its input. For example, it is good enough if the
underlying $C_{m}$-equivariant orthogonal spectrum indexed on
$U_{C_{m}}$ is cofibrant, and we note that if $A$ satisfies this, then
so does $B(\bT,\bT,A)$.  In this case then, we get a weak equivalence
between both relative norms on $A$.

We note that in the case $m=1$, the theorem and corollary above give
an equivariant comparison between factorization homology and $THH$.

We now begin the proof of the theorem.  For the first display, the
argument amounts to little more than defining terms.  The functor
$\bar{\mathrm E}$ in the statement is a simplification of
$\bar\bE_{S^{1}/C_{m}}$ along the lines that the little disk operad is
a simplification of the disk embedding spaces.  

\begin{cons}\label{cons:E}
Let
$E(1)$ be the set of ordered pairs $(\zeta,r)$ with $\zeta \in S^{1}$
and $r\in (0,\pi/m]$. Such an ordered pair specifies an element of
$\oE_{S^{1}/C_{m}}(1)$ where
\begin{itemize}
\item The embedding $f\colon D\to S^{1}/C_{m}$ is the map $t\mapsto
[e^{irt}\zeta]$ (writing $[\alpha]$ for the image in $S^{1}/C_{m}$ of
an 
element $\alpha \in S^{1}$, and thinking of $S^{1}$ as the unit
complex numbers). 
\item The map of $C_{m}$-frame bundles $F_{C_{m}}D\to
f^{*}F_{C_{m}}S^{1}/C_{m}$ is determined by the section $t\mapsto
e^{irt}\zeta$.
\item The map of $C_{m}$-frame bundles lies over the map of frame
bundles that is the identity (under the identification of the tangent
space of $S^{1}/C_{m}$ with $\bR$ given above) and the derivative is
multiplication by $r$ in the fiber of each point $t\in D$; we use the
homotopy $s\mapsto r^{s}$.
\end{itemize}
We let $E(n)$ be the subspace of $E(1)^{n}$ where the images of the
embeddings do not overlap.  We then get an inclusion $E(n)\to
\oE_{S^{1}/C_{m}}(n)$. We write $E$ for the collection $E(n)$, $n\geq 0$.
\end{cons}

The following is clear from construction.

\begin{prop}
The right $\oD^{\bR}_{C_{m}}$-action on $\oE_{S^{1}/C_{m}}$ restricts
to define a right $\oD^{\bR}_{C_{m}}$-action on $E$.
\end{prop}

The map $E(n)\to \oE_{S^{1}/C_{m}}(n)$ is $S^{1}\times (\Sigma_{n}\wr
C_{m})^{\op}$-equivariant (equivariant for both the left action of $S^{1}$
and the right action of $\Sigma_{n}\wr C_{m}$).
Since $C_{n}\to \GL(\bR)$ is the trivial map, all embeddings in
$\oE_{S^{1}/C_{m}}(n)$ are oriented. The exponential map from
$\bR$ to the oriented transformations in $\GL(\bR)$ is an isomorphism,
and the following proposition is then easy using linear homotopies.

\begin{prop}\label{prop:EtooE}
For each $n$, the inclusion of $E(n)$ in $\oE_{S^{1}/C_{m}}(n)$ is a
$S^{1}\times (\Sigma_{n}\wr C_{m})^{\op}$-equivariant homotopy equivalence.
\end{prop}

We define the functor $\bar{\mathrm E}$ from $C_{m}$-equivariant
orthogonal spectra to $S^{1}$-equivariant orthogonal spectra as in
Construction~\ref{cons:compressed}: let 
\[
\bar{\mathrm E}X=\bigvee_{n\geq 0}E(n)_{+}\sma_{\Sigma_{n}\wr C_{m}}X^{(n)}.
\]
We then get a monadic bar construction $B(\bar{\mathrm E},\bar\bD,-)$
with input $C_{m}$-equivariant $\oD_{\bR}$-algebras in orthogonal
spectra and output $S^{1}$-equivariant orthogonal spectra.  The map of
right $\oD^{\bR}_{C_{m}}$-spaces 
$E\to \oE_{S^{1}/C_{m}}$ induces a map of bar constructions, and
Proposition~\ref{prop:EtooE} then implies the following proposition.

\begin{prop}\label{prop:zig1}
For any $C_{m}$-equivariant $\oD_{\bR}$-algebra $X$, the map of
monadic bar constructions
\[
B\subdot(\bar{\mathrm{E}},\bar \bD,X)\to 
B\subdot(\oE_{S^{1}/C_{m}},\bar \bD,X)
\]
is on each level a natural $S^{1}$-equivariant homotopy equivalence. 
\end{prop}

The idea for $\bar{\mathrm{E}}^{c}$ and $\bar\bD^{c}$ is to expand $E$
and $\oD_{\bR}$ to allow the radius of the disk images to go to zero,
while retaining the correct overall homotopy type for these embedding
spaces.  This is easiest to first define in the non-symmetric context
and then put the symmetries back in.  To do this, let $u\oD_{\bR}(n)$
be the subspace of $\oD_{\bR}(n)$ consisting of those $n$-tuples
$(\lambda_{1},\dotsc,\lambda_{n})$ such that
\[
\lambda_{1}(0)<\dotsb<\lambda_{n}(0);
\]
then $\oD_{\bR}(n)\iso u\oD_{\bR}\times \Sigma_{n}$.  The operadic
composition map preserves these components (making $u\oD_{\bR}$ a
``non-$\Sigma$ operad'') and the isomorphism $\oD_{\bR}(-)\iso
u\oD_{\bR}(-)\times \Sigma_{-}$ is an isomorphism of operads where the operadic
composition (and permutation action) on the symmetric groups is the
standard one, defining the operad for associative monoids $\oA$.
We extend this operad as follows.

\begin{cons}
Let $u\oD^{c}(1)$ be the subspace of points $(v,r)\in \bR\times [0,1]$
such that the affine transformation 
\[
\lambda \colon t\mapsto v+rt
\]
sends $D=[-1,1]\subset \bR$ into $D$.  Let $u\oD^{c}(n)$ be the subset
of $(u\oD^{c}(1))^{n}$ consisting of those $n$-tuples
$((v_{1},r_{1}),\dotsc,(v_{n},r_{n}))$ that satisfy
\[
v_{1}\leq \dotsb \leq v_{n}
\]
and whenever $v_{j}+r_{j}>v_{j+1}-r_{j+1}$ we have $r_{j}=r_{j+1}=0$;
in terms of the affine transformations, whenever the images of two
overlap, they are both constant to the same point.
Let $\oD^{c}(n)=u\oD^{c}(n)\times \Sigma_{n}$, with the operadic
multiplication induced by diagonal composition and block sum of
permutations:  composition takes the element corresponding to
\[
(((\lambda_{1},\dotsc,\lambda_{n}),\sigma),((\mu_{1,1},\dotsc),\tau_{1}),\dotsc,((\dotsc,\mu_{n,j_{n}}),\tau_{n}))
\]
in $\oD^{c}(n)\times \oD^{c}(j_{1})\times \dotsb\times \oD^{c}(j_{n})$
to the element corresponding to
\[
((\lambda_{1}\circ \mu_{\sigma^{-1}(1),1},\dotsc,\lambda_{1}\circ \mu_{\sigma^{-1}(1),j_{\sigma^{-1}(1)}},\dotsc,
\lambda_{n}\circ \mu_{\sigma^{-1}(n),j_{\sigma^{-1}(n)}}),\sigma \circ (\tau_{1}\oplus \dotsb \oplus \tau_{n}))
\]
in $\oD^{c}(j)$. 
\end{cons}

The inclusion of $\oD_{\bR}$ in $\oD^{c}$ is a map of operads, and
we also get a map of operads $\oA\to \oD^{c}$ from the operad
$\oA$ for associative monoids: the map 
\[
\oA(n)=\Sigma_{n}\to u\oD^{c}(n)\times \Sigma_{n}=\oD^{c}(n)
\]
sends $\sigma\in \Sigma_{n}$ to $((0,0),\dotsc,(0,0),\sigma)$.
Because $u\oD_{\bR}(n)$ and $u\oD^{c}(n)$ are contractible for all
$n$, we get the following proposition. 

\begin{prop}\label{prop:DDcA}
The maps of operads $\oD_{\bR}\to \oD^{c}\from \oA$ induce
$\Sigma_{n}$-equivariant homotopy equivalences
\[
\oD_{\bR}(n)\to \oD^{c}(n)\from \oA(n)
\]
for all $n$.
\end{prop}

We can do something similar for $E(n)$: let $uE(n)$ be the subspace of
$E(n)$ of those $n$-tuples
$((\zeta_{1},r_{1}),\dotsc,(\zeta_{n},r_{n}))$ such that for each $i$,
the counterclockwise path from $[\zeta_{j}]$ to $[\zeta_{j+1}]$ does
not pass though any of the other $[\zeta_{k}]$, that is, starting with
$[\zeta_{1}]$, the points $[\zeta_{j}]$ are numbered in strictly
counterclockwise order. This has a canonical action of
$Z_{n}<\Sigma_{n}$ (for $Z_{n}=\langle (1\dotsb n)\rangle$) and
$uE(n)$ is an $S^{1}\times (Z_{n}\wr C_{m})^{\op}$-subspace of $E(n)$.
Moreover, $E(n)$ is isomorphic as an $S^{1}\times (\Sigma_{n}\wr
C_{m})^{\op}$-space to $uE(n)\times_{C_{n}}\Sigma_{n}$.  The right
action of $\oD^{\bR}_{C_{m}}$ on $E$ restricts to give maps of the
form
\[
uE(n) \times u\oD_{\bR}(j_{1})\times \dotsb \times uD_{\bR}(j_{n})
\to uE(j)
\]
which are $S^{1}\times (C_{m}^{n})^{\op}$-equivariant where $C_{m}^{n}$ acts
by diagonal blocks on the right.  Writing $u\oD^{\bR}_{C_{m}}$ for the
non-$\Sigma$ version of $\oD_{C_{m}}^{\bR}$, $u\oD_{\bR}\rtimes
C_{m}$, the corresponding action map
\[
uE(n) \times u\oD^{\bR}_{C_{m}}(j_{1})\times \dotsb \times uD^{\bR}_{C_{m}}(j_{n})
\to uE(j)
\]
is $S^{1}\times (C_{m}^{j})^{\op}$-equivariant.  We have a further
cyclic $Z_{n}$ invariance of the following form: for 
\[
(f,g_{1},\dotsc,g_{n})\in uE(n) \times 
u\oD_{\bR}(j_{1})\times \dotsb \times uD_{\bR}(j_{n})
\to uE(j),
\]
$\alpha \in Z_{n}$, and $\alpha_{j_{1},\dotsc,j_{n}}\in \Sigma_{j}$ the
permutation that cycles the blocks
$j_{1},\dotsc,j_{n}$ by $\alpha$, the following compositions are equal:
\[
(f\alpha)\circ (g_{1},\dotsc,g_{n})
=(f\circ (g_{\alpha^{-1}(1)},\dotsc,g_{\alpha^{-1}(n)}))\alpha_{j_{1},\dotsc,j_{n}}
\in uE(j)
\]
We note that
$\alpha_{j_{1},\dotsc,j_{n}}\in Z_{j}<\Sigma_{j}$; for example, if
$\alpha$ is the cycle $(1\dotsb n)$ then $\alpha_{j_{1},\dotsc,j_{n}}$
is $(1\dotsb j)^{j_{n}}$.  Adding the full symmetric group
symmetries back in, these action maps induce the
$\oD^{\bR}_{C_{m}}$-action maps.

The $Z_{n}$-action on $uE(n)$ adds an extra complication to
constructing the extension $uE^{c}(n)$.  Since $Z_{1}$ is the trivial
group, no issues arise at the $1$-ary level, and we can take
$uE^{c}(1)=E^{c}(1)$ to be the set or ordered pairs $(\zeta,r)$ with
$\zeta \in S^{1}$ and $r\in [0,\pi/m]$.  The problem arises at the
$2$-level: consider the elements
\[
((1,r),(e^{2ir},r)),\quad ((e^{2ir},r),(1,r))\quad \in uE(2).
\]
As $r$ goes to zero, these need to converge to different elements of
$uE^{c}(2)$, and so we cannot just take $uE^{c}(2)$ to be the obvious
subspace of $uE^{c}(1)^{2}$.  Instead, we note that when $r$ is small,
the center points $1$ and $e^{2ir}$ are close together and we can
interpret the point $1$ as being ``first'' in the counter-clockwise
order; we can identify it as first because traveling only
counter-clockwise, most of the circle has to be traversed to reach it
from the other point.  We use this idea to redefine $uE(n)$ in a way
that we extends to allow the size of
the disk images to be zero.

Define $\theta_{j}\colon C(n,S^{1}/C_{m})\to (0,2\pi/m)$ to be the
(continuous) function that takes a configuration
$(x_{1},\dotsc,x_{n})$ to the length of the counter-clockwise arc from
$x_{j}$ to $x_{j+1}$ (for $j<n$) or from $x_{n}$ to $x_{1}$ (for
$j=n$).  We
define $\theta_{j}$ on $E(n)$ and $uE(n)$ using the \term{center point
map} $E(n)\to C(n,S^{1}/C_{m})$ that takes an element
$((\zeta_{1},r_{1}),\dotsc,(\zeta_{n},r_{n}))$ to the configuration
given by the center points of the disk images
$([\zeta_{1}],\dotsc,[\zeta_{n}])$.  Since for elements of $uE(n)$,
the center points occur cyclically in the counter-clockwise direction, 
we have that the sum of the lengths always adds up to the
circumference of the circle,
\[
\theta_{1}|_{uE(n)}+\dotsb+\theta_{n}|_{uE(n)} \equiv 2\pi/m.
\]
Now let $uE'(n)$ be the subspace of $E(1)^{n}\times (0,2\pi/m)^{n}$
consisting of the points 
\[
((\zeta_{1},r_{1}),\dotsc,(\zeta_{n},r_{n})),(\phi_{1},\dotsc,\phi_{n}))
\]
such that
\begin{itemize}
\item Starting at $[\zeta_{1}]$, the points
$[\zeta_{1}],\dotsc,[\zeta_{n}]$ occur in counter-clockwise order;
\item The intervals $t\mapsto [e^{\pi i r_{j}t}\zeta_{j}]$, $t\in (-1,1)$ do not overlap;
\item $[\zeta_{j+1}]=[e^{\pi i \phi_{j}}\zeta_{j}]$ for
$j=0,\dotsc,n-1$ where $\zeta_{0}:=\zeta_{n}$ and
$\phi_{0}:=\phi_{n}$; and
\item $\phi_{1}+\dotsb+\phi_{n}=2\pi/m$.
\end{itemize}
Then the projection $uE'(n)\to uE(n)$ and the map $uE(n)\to uE'(n)$
given by the inclusion and $\theta_{1},\dotsc,\theta_{n}$ 
are inverse homeomorphisms.  Moreover, if we let $Z_{n}$ act on
$(0,2\pi/m)^{n}$ by permuting coordinates (and let $S^{1}$ and $C_{m}$
act trivially on $(0,2\pi/m)^{n}$), then these homeomorphisms are
$S^{1}\times (Z_{n}\wr C_{m})^{\op}$-equivariant.  The effect of the
action of $u\oD_{\bR}$ on the new $\phi$ coordinates is straight-forward
but tedious to describe; to avoid unnecessary redundancy, we just
write it out for the extension $uE^{c}$ below.

\begin{cons}
Let $uE^{c}(0)$ be a point and $uE^{c}(1)$ the set or ordered pairs $(\zeta,r)$ with
$\zeta \in S^{1}$ and $r\in [0,\pi/m]$. For $n>1$, 
let $uE'(n)$ be the subspace of $E^{c}(1)^{n}\times [0,2\pi/m]^{n}$
consisting of the points 
\[
((\zeta_{1},r_{1}),\dotsc,(\zeta_{n},r_{n})),(\phi_{1},\dotsc,\phi_{n}))
\]
such that
\begin{itemize}
\item Starting at $[\zeta_{1}]$, the points
$[\zeta_{1}],\dotsc,[\zeta_{n}]$ occur in counter-clockwise order;
\item If for some $j<k\in \{1,\dotsc,n\}$ the intervals $t\mapsto
[e^{\pi i r_{j}t}\zeta_{j}]$ and $t\mapsto [e^{\pi i r_{k}t}\zeta_{k}]$, $t\in
(-1,1)$ overlap, then $r_{j}=r_{k}=0$;
\item $[\zeta_{j+1}]=[e^{\pi i \phi_{j}}\zeta_{j}]$ for
$j=0,\dotsc,n-1$ where $\zeta_{0}:=\zeta_{n}$ and
$\phi_{0}:=\phi_{n}$; and
\item $\phi_{1}+\dotsb+\phi_{n}=2\pi/m$.
\end{itemize}
We have an $S^{1}\times (Z_{n}\wr C_{m})^{\op}$-action with the left $S^{1}$
action diagonally on the $\zeta_{j}$, the right $C_{m}$ actions
individually on the $\zeta_{j}$, and the $Z_{n}$ action permuting the
indexes on the $\zeta_{j}$, $r_{j}$, and $\phi_{j}$.  

We define the action map
\[
uE^{c}(n)\times u\oD^{c}(j_{1})\times \dotsb \times u\oD^{c}(j_{n})
\to uE^{c}(j)
\]
($j=j_{1}+\dotsb+j_{n}$) as follows.  For 
\begin{gather*}
((\zeta_{1},r_{1}),\dotsc,(\zeta_{n},r_{n}),(\phi_{1},\dotsc,\phi_{n}))\in uE^{c},\\
((v_{i,1},s_{i,1}),\dotsc,(v_{i,j_{i}},s_{i,j_{i}}))\in u\oD^{c}(j_{i}),
\end{gather*}
the resulting element of $uE^{c}(j)$, 
\[
((\xi_{1},t_{1}),\dotsc,(\xi_{j},t_{j}),(\psi_{1},\dotsc,\psi_{j}))
\]
is given as follows.  For $\ell\in \{1,\dotsc,j\}$ define $j(\ell)$ to
be the smallest integer such that 
\[
j_{1}+\dotsb+j_{j(\ell)}\geq\ell
\]
(so $j(\ell)=1$ for $\ell\in \{1,\dotsc,j_{1}\}$, $j(\ell)=2$ for
$\ell\in \{j_{1}+1,\dotsc,j_{1}+j_{2}\}$, etc.), and define 
\[
k(\ell)=\ell-(j_{1}+\dotsb+j_{j(\ell)-1})
\]
(where we understand the parenthetical sum as $0$ when $j(\ell)=1$).
We have made this definition so that the $\ell$ index in the codomain
corresponds to the $j(\ell),k(\ell)$ index in the domain, and we take
\begin{gather*}
\xi_{\ell}=e^{ir_{j(\ell)}v_{j(\ell),k(\ell)}}\zeta_{j(\ell)}\\
t_{\ell}=r_{j(\ell)}s_{j(\ell),k(\ell)};
\end{gather*}
in other words, in terms of the corresponding maps of the interval, we
do the usual diagonal composition.  For the parameters $\psi_{\ell}$,
we take 
\[
\psi_{\ell}=\begin{cases}
r_{j(\ell)}(v_{j(\ell),k(\ell)+1}-v_{j(\ell),k(\ell)})
  &\text{if }\ell<j\text{ and }j(\ell+1)=j(\ell)\\
\phi_{j(\ell)}-r_{j(\ell)}v_{j(\ell),k(\ell)}
  +r_{j(\ell)+1}v_{j(\ell)+1,1}
  &\text{if }\ell<j\text{ and }j(\ell+1)=j(\ell)+1\\
\phi_{n}-r_{n}v_{n,j_{n}}
  +r_{1}v_{1,1}
  &\text{if }\ell=j;
\end{cases}
\]
an easy check shows that this defines an element of $uE^{c}(j)$.

We define $E^{c}(n)$ to be the $S^{1}\times (\Sigma_{n}\wr
C_{m})^{\op}$-space $E^{c}(n)\times_{Z_{n}}\Sigma_{n}$.  The cyclic
permutation action on the $uE^{c}(n)$ has the same compatibility with
composition as on the $uE(n)$, and we get a right action of $\oD^{c}$
on $E^{c}$ generalizing the formulas above for the right action of
$\oD_{\bR}$ on $E$.
\end{cons}

As discussed above, the functions $\theta_{j}$ on $uE(n)$ fill in
parameters $\phi_{j}$ to define a map $uE(n)\to uE^{c}(n)$ and hence a
map $E(n)\to E^{c}(n)$, which are easily seen to be inclusions.  The
latter map is $S^{1}\times (\Sigma_{n}\wr C_{m})^{\op}$-equivariant;
we show it is an $S^{1}\times (\Sigma_{n}\wr C_{m})^{\op}$-equivariant
homotopy equivalence.

\begin{prop}\label{prop:EtoEc}
The map $E(n)\to E^{c}(n)$ is an $S^{1}\times (\Sigma_{n}\wr
C_{m})^{\op}$-equivariant homotopy equivalence.
\end{prop}

\begin{proof}
There is nothing to show in the case $n=0$ and the case $n=1$ is
clear.  For $n\geq 2$, we can identify $E(n)$ with its homeomorphic
image, which consists of the elements where (in our usual notation)
none of the $r_{j}$'s are zero.  Let $X$ denote the subspace of
$E^{c}(n)$ where all the $r_{j}$'s are zero.  We have an obvious
equivariant deformation retraction of $E(n)$ on to $X$, which
induces an equivariant homotopy equivalence between $E(n)$ and the
subspace $X_{0}$ of $X^{c}$ where the $[\zeta_{j}]$ are all distinct
elements of $S^{1}/C_{m}$.  (The center point map gives an 
equivariant homeomorphism from $X$ to the $C_{m}$-framed
configuration space $C_{C_{m}}(n,S^{1}/C_{m})$ described in the
paragraph following Definition~\ref{defn:relnorm}, with the functions
$\theta$ above inducing the inverse.) We can
equivalently describe $X_{0}$ as the subspace where all the $\phi_{j}$
are positive; let $X_{k}\subset X$ be the subspace where at most $k$
of the $\phi_{j}$ are zero.  Then $X=X_{n}$.  Consider the
equivariant self-homotopy of $X$ that at time $t$ sends the
element represented by
\[
((((\zeta_{1},0),\dotsc,(\zeta_{n},0)),(\phi_{1},\dotsc,\phi_{n})),\sigma)
\]
to the element represented by
\[
((((e^{it\phi_{1}/2}\zeta_{1},0),\dotsc,(e^{it\phi_{n}/2}\zeta_{n},0)),
(\phi_{1}(1-t/2)+\phi_{2}t/2,\dotsc,\phi_{n}(1-t/2)+\phi_{1}t/2)),\sigma).
\]
This starts at the identity and ends at an endomorphism $f$ of $X$.
The endomorphism $f$ sends $X_{k}$ into $X_{k-1}$ for
$k>0$.  The $n$th iterate then sends $X$ into $X_{0}$ and is
evidently an equivariant homotopy inverse to the inclusion.
\end{proof}

Let $\bar{\mathrm E}^{c}$ be the functor from $C_{m}$-equivariant
orthogonal spectra to $S^{1}$-equivariant orthogonal spectra defined
by
\[
\bar{\mathrm E}^{c}X
=\bigvee_{n\geq 0} E^{c}(n)_{+}\sma_{\Sigma_{n}\wr C_{m}}X^{(n)}
\iso 
\bigvee_{n\geq 0} uE^{c}(n)_{+}\sma_{Z_{n}\wr C_{m}}X^{(n)}
\]
and let $\bar \bD^{c}$ be the monad in $C_{n}$-equivariant orthogonal
spectra associated to the $C_{n}$-equivariant operad $\oD^{c}$, 
\[
\bD^{c}X=\bigvee_{n\geq 0} \oD^{c}(n)_{+}\sma X^{(n)}.
\]
The right action of $\oD^{c}$ on $E^{c}$ descends to give a right action of the monad $\bD^{c}$ on
$\bar{\mathrm{E}}^{c}$, and the inclusion of $E$ in $E^{c}$ is
compatible with the actions in the sense that the diagram
\[
\xymatrix@-1pc{%
\bar{\mathrm{E}}\bar\bD\ar[r]\ar[d]
&\bar{\mathrm{E}}\ar[d]\\
\bar{\mathrm{E}}^{c}\bar\bD^{c}\ar[r]
&\bar{\mathrm{E}}^{c}
}
\]
of functors from $C_{m}$-equivariant orthogonal spectra to
$S^{1}$-equivariant orthogonal spectra precisely commutes. This induces a map of
monadic bar constructions
\[
B(\bar{\mathrm{E}},\bar\bD,-)\to
B(\bar{\mathrm{E}}^{c},\bar\bD^{c},-).
\]
Propositions~\ref{prop:DDcA} and~\ref{prop:EtoEc} now prove the following proposition.

\begin{prop}\label{prop:zig2}
For any $C_{m}$-equivariant $\oD_{\bR}$-algebra $X$, the map of
monadic bar constructions
\[
B\subdot(\bar{\mathrm{E}},\bar \bD,X)\to 
B\subdot(\bar{\mathrm{E}}^{c},\bar \bD^{c},X)
\]
is on each level a natural $S^{1}$-equivariant homotopy equivalence. 
\end{prop}

For the last zigzag, let $uC^{c}(n)$ be the subspace of $uE^{c}(n)$
where the $r_{j}$'s are all zero, and let
$C^{c}(n)=uC^{c}\times_{Z_{n}}\Sigma_{n}$.  Viewing $C^{c}(n)$ as a
subspace of $E^{c}(n)$, it inherits a $S^{1}\times (\Sigma_{n}\wr
C_{m})^{\op}$-action and the inclusion of $C^{c}(n)$ in $E^{c}(n)$ is
an equivariant homotopy equivalence.  The right $\oD^{c}$-action on
$E^{c}$ restricts to $C^{c}$, and using the map of operads $\oA\to
\oD^{c}$, we get a right action of $\oA$ on $\oC^{c}$, making the
following diagram precisely commute
\[
\xymatrix@-1pc{%
\bar{\mathrm{C}}^{c}\bT\ar[r]\ar[d]
&\bar{\mathrm{C}^{c}}\ar[d]\\
\bar{\mathrm{E}}^{c}\bar\bD^{c}\ar[r]
&\bar{\mathrm{E}}^{c}
}
\]
where $\bT$ is the free associative algebra monad (the monad
associated to the operad $\oA$) and
$\bar{\mathrm{C}}^{c}$ denotes the functor from
$C_{m}$-equivariant orthogonal spectra to $S^{1}$-equivariant
orthogonal spectra
\[
\bar{\mathrm{C}}^{c}X=
\bigvee_{n\geq 0} C^{c}(n)_{+}\sma_{\Sigma_{n}\wr C_{m}}X^{(n)}
\iso 
\bigvee_{n\geq 0} uC^{c}(n)_{+}\sma_{Z_{n}\wr C_{m}}X^{(n)}
\]
We get a corresponding map of bar constructions and the following
proposition is now clear.

\begin{prop}\label{prop:zig3}
For any $C_{m}$-equivariant associative ring orthogonal spectrum $R$,
the map of monadic bar constructions
\[
B\subdot(\bar{\mathrm{C}}^{c},\bT,R)\to 
B\subdot(\bar{\mathrm{E}}^{c},\bar \bD^{c},R)
\]
is on each level a natural $S^{1}$-equivariant homotopy equivalence. 
\end{prop}

This completes the proof of the assertion about the first display in
Theorem~\ref{thm:circle}. For the isomorphism in the second display,
we use the following construction.

\begin{cons}
For $R$ a $C_{m}$-equivariant associative ring orthogonal spectrum.
Define the $S^{1}$-equivariant orthogonal spectrum
$\bar{\mathrm{C}}^{c}\otimes_{\bT}X$ to be the (point-set) coequalizer
\[
\xymatrix@C-.75pc{%
\bar{\mathrm{C}}^{c}\bT R\ar@<-.5ex>[r]\ar@<+.5ex>[r]
&\bar{\mathrm{C}}^{c} R\ar[r]
&\bar{\mathrm{C}}^{c}\otimes_{\bT}R
}
\]
with one map 
$\bar{\mathrm{C}}^{c}\bT R\to \bar{\mathrm{C}}^{c} R$
induced by the right $\bT$-action on $\bar{\mathrm{C}}^{c}$ and the
other induced by the left $\bT$-action on $R$.
\end{cons}

Since the functors $\bar{\mathrm{C}}^{c}$ and $\bT$ commute with
geometric realization, we have a natural isomorphism
\[
\bar{\mathrm{C}}^{c}\otimes_{\bT}B(\bT,\bT,R)\iso
B(\bar{\mathrm{C}}^{c},\bT,R).
\]
The proof of Theorem~\ref{thm:circle} is therefore completed by the
verification of the following theorem.

\begin{thm}\label{thm:cycbar}
There is a natural isomorphism 
\[
\bar{\mathrm{C}}^{c}\otimes_{\bT}(-)\iso N^{cyc,C_{m}}_{\sma}(-)
\]
of functors from $C_{m}$-equivariant associative ring orthogonal
spectra to $S^{1}$-equivariant orthogonal spectra.
\end{thm}

Before explaining the isomorphism, we begin with a brief review of the
functor $N^{cyc,C_{m}}_{\sma}$ from $C_{m}$-equivariant associative
ring orthogonal spectra to $S^{1}$-equivariant orthogonal spectra.
Non-equivariantly, $N^{cyc,C_{m}}_{\sma}$ is a variant of the cyclic
bar construction: it is the geometric realization of the simplicial
object with $q$th object the $(q+1)$th smash power
\[
N^{cyc,C_{m}}_{\sma,q}R=R^{(q+1)}
\]
with degeneracy $s_{i}$ induced by the inclusion of the identity in
the $(i+1)$th position and face map $d_{i}$ for $i=0,\dotsc,q-1$, the
multiplication in positions $i+1$ and $i+2$.  The last face map
$d_{q+1}$ cycles the last position around to the front acts on it by
the generator $e^{2\pi i/m}\in C_{m}<S^{1}<\bC^{\times}$ and then
multiplies the (new) $1$st and $2$nd positions.  Then
$d_{q+1}=d_{0}\circ \tau_{q}$ where $\tau_{q}$ is the operation on
$R^{(q+1)}$ that cycles the factors and then applies $e^{2\pi i/m}$ in
the first factor.  The face, degeneracy, and $\tau$ operators satisfy the
relations
\begin{equation}\label{eq:cycm}
\begin{aligned}
&\qquad\qquad \tau_{q}^{m(q+1)}=\id\\
d_{0}\tau_{q}&=d_{q}&s_{0}\tau_{q}&=\tau_{q+1}^{2}s_{q}\\
d_{i}\tau_{q}&=\tau_{q-1}d_{i-1}&
s_{i}\tau_{q}&=\tau_{q+1}s_{i-1}&&(1\leq i\leq q)\\
\end{aligned}
\end{equation}
(in addition to the simplicial identities relating just the faces and
degeneracies). As in~\cite[\S1]{BHM} (but with slightly different
indexing conventions), this implicitly defines a category
$\LLambda^{\op}_{m}$, generalizing Connes' cyclic category, such that
$N^{cyc,C_{m}}_{\sma,\bullet}$ is a functor from $\LLambda^{\op}_{m}$
to orthogonal spectra.  We use the terminology \term{$m$-cyclic
orthogonal spectrum} (or \term{$m$-cyclic space}) for a functor from
$\LLambda_{m}^{\op}$ to orthogonal spectra (or spaces).

Let $\Lambda_{m}[q]$ denote the geometric realization of the
representable object
\[
\Lambda_{m}[q]\subdot=\Lambda_{m}(\bullet,q).
\]
As $q$, varies $\Lambda_{m}[q]$ is a functor from $\LLambda_{m}$ to
spaces.  Then (as in~\cite[1.8]{BHM}), for any
$\LLambda_{m}^{\op}$-object $X\subdot$ (in spaces or orthogonal
spectra), the inclusion of the simplex category $\DDelta$ in
$\LLambda_{m}$ induces an isomorphism from usual geometric realization of
$X\subdot$ to an $m$-cylic realization given as the coend over
$\LLambda_{m}$ of $X\subdot \sma \Lambda_{m}[\bullet]_{+}$ (in the case of
orthogonal spectra).  The spaces $\Lambda_{m}[q]$ have a natural (in
$q\in \LLambda_{m}^{\op}$) action of the circle $S^{1}$;
see~\cite[1.6]{BHM}.  This gives the geometric realization of any
$m$-cyclic orthogonal spectrum (or space) a natural $S^{1}$-action.

To be precise, $\Lambda_{m}[q]$ is isomorphic to the space 
\[
\bR/m\bZ \times \Delta[q].
\]
Writing an element as $(r+m\bZ,t_{0},\dotsc,t_{q})$ where $r\in \bR$
and $t_{i}\geq 0$, $t_{0}+\dotsb+t_{q}=1$, the circle acts by
\[
e^{i \theta}\cdot (r+m\bZ,t_{0},\dotsc,t_{q})
=(r+\theta/(2\pi/m)+m\bZ,t_{0},\dotsc,t_{q}).
\]
As a functor of $\LLambda_{m}$, the face and degeneracy maps act in
the usual manner on the simplices and the twist $\tau_{q}$ acts by
\[
\tau_{q}(r+m\bZ,t_{0},\dotsc,t_{q})
=(r-t_{q}+m\bZ,t_{q},t_{0},\dotsc,t_{q-1}).
\]
(This is~\cite[1.6]{BHM} adjusted for our indexing convention.)

The spaces $\Lambda_{m}[q]$ are closely related to the spaces
$uC^{c}(n)$. To simplify notation, we write a typical 
element of $uC^{c}(n)$ as
\[
(\zeta_{1},\dotsc,\zeta_{n},\phi_{1},\dotsc,\phi_{n})
\]
for $\zeta_{i}\in S^{1}$, $\phi_{i}\in [0,2\pi/m]$ (dropping the
$r_{j}=0$ from the notation we used above and flattening
parentheses).  Let $\upsilon_{m,n} \in Z_{n}\wr C_{m}$ denote the
element 
\[
((1\dotsb n);1,\dotsc,1,e^{-2\pi i/m})\in Z_{n}\ltimes C_{m}^{n}
=Z_{n}\wr C_{m}.
\]
Then $\upsilon_{m,n}$
generates a cyclic subgroup of order $mn$ in $Z_{n}\wr C_{m}$ that
acts on $X^{(n)}$ (for a $C_{m}$-equivariant orthogonal spectrum $X$)
by acting by $e^{-2\pi i/m}$ on the last factor and then cycling it to
the first position. The precise relationship between the spaces
$\Lambda_{m}[q]$ and $uC^{c}(n)$ is as follows.

\begin{prop}
The map $\Lambda_{m}[q]\to uC^{c}(q+1)$ that sends
$(r+m\bZ,t_{0},\dotsc,t_{q})$ to 
\[
(e^{(2\pi/m)ir}, e^{(2\pi/m)i(r+t_{0})},\dotsc,e^{(2\pi/m)i(r+t_{0}+\dotsb+t_{q-1})},
(2\pi/m)t_{0},\dotsc,(2\pi/m)t_{q})
\]
induces a $S^{1}\times (Z_{q+1}\wr C_{m})^{\op}$-equivariant
isomorphism 
\[
\Lambda_{m}[q]\times_{C_{m(q+1)}}(Z_{q+1}\wr C_{m})\to uC^{c}(q+1)
\]
(where on the left, we are using the isomorphism $C_{m(q+1)}\iso
\langle u_{m,q+1}\rangle \subset Z_{q+1}\wr C_{m}$ sending the
generator $e^{2\pi i/(m(q+1))}$ to $\upsilon_{m,q+1}$).
\end{prop}

\begin{proof}
The displayed formula for the map $\Lambda_{m}[q]\to uC^{c}(q+1)$ is
clearly well-defined and $S^{1}$-equivariant; moreover, it is
equivariant for the right action of $C_{m(q+1)}$ on $\Lambda_{m}[q]$
and the $\langle \upsilon_{m,q+1}\rangle$ action on $uC^{c}(q+1)$ under
the given isomorphism since
\[
e^{-2\pi i/m}e^{(2\pi/m)i(r+t_{0}+\dotsb+t_{q-1})}=e^{(2\pi/m)i(r-t_{q})}
\]
($-1+t_{0}+\dots t_{q-1}=-t_{q}$).  The map
$\Lambda_{m}[q]\times_{C_{m(q+1)}}(Z_{q+1}\wr C_{m})\to uC^{c}(q+1)$
is therefore well-defined and $S^{1}\times (Z_{q+1}\wr
C_{m})^{\op}$-equivariant.  It is a continuous bijection of compact
Hausdorff spaces and therefore an isomorphism. 
\end{proof}

Using the isomorphism above, we get a well defined map
\begin{multline}\label{eq:CtoL}
\hspace{2em}
\bar{\mathrm C}^{c}R\iso \bigvee_{n\geq 0}uC^{c}(n)_{+}\sma_{Z_{n}\wr C_{m}}R^{(n)}\\
\to N^{cyc,C_{m}}_{\sma}(R) \iso \bigg(\bigvee_{q\geq 0} \Lambda_{m}[q]_{+}\sma R^{(q+1)}\bigg)\bigg/\sim
\hspace{4em}
\end{multline}
sending the $n=0$ summand $uC^{c}(0)_{+}\sma R^{(0)}\iso \bS$ by the
inclusion of $\{1\}_{+}\sma \bS$ in $\Lambda_{m}[0]\sma R^{(0)}\iso
S^{1}_{+}\sma \bS$, and for $n>0$, sending the $n$th summand through the $q=n-1$
summand using the isomorphism 
\[
uC^{c}(q+1)_{+}\sma_{Z_{q+1}\wr C_{m}}R^{(q+1)}\iso \Lambda_{m}[q]_{+}\sma_{C_{m(q+1)}}R^{(q+1)}
\]
implied by the previous proposition.  This is well defined because on
the right hand side, the $C_{m(q+1)}$-actions on
$\Lambda_{m}[q]_{+}\sma R^{(q+1)}$ are coequalized as part of the coend.
It is obvious that this map is $S^{1}$-equivariant on the $n$th summand
for $n>0$, but it is also $S^{1}$-equivariant on the $0$th summand (as
a consequence of the fact that on the image of $\bS$ in
$N^{cyc,C_{m}}_{\sma,0}$, $\tau_{1}s_{0}=s_{0}$).  

For the coequalizer forming $\bar{\mathrm C}^{c}\otimes_{\bT}R$ from $\bar{\mathrm C}^{c}R$,
\[
\bar{\mathrm C}^{c}\otimes_{\bT}R \iso \bigg(\bigvee_{n\geq 0}uC^{c}(n)_{+}\sma_{Z_{n}\wr C_{m}}R^{(n)}\bigg)\bigg/\sim,
\]
the equivalence relation induced by the action of $\oA$ on $C^{c}$ and
$R$ can be written in terms of faces, degeneracies, and twists in the
$m$-cyclic object $N^{cyc,C_{m}}_{\sma,\bullet}$ and relations
involving the the unit $\bS\to R$.  As a consequence, the map
$\bar{\mathrm C}^{c}R\to N^{cyc,C_{m}}_{\sma}R$ induces a
map $\bar{\mathrm C}^{c}\otimes_{\bT}R\to N^{cyc,C_{m}}_{\sma}R$ and
we get the following proposition.

\begin{prop}\label{prop:CtoL}
The map of~\eqref{eq:CtoL} induces a natural
transformation 
\[
\bar{\mathrm C}^{c}\otimes_{\bT}(-)\to
N^{cyc,C_{m}}_{\sma}(-)
\]
of functors from $C_{m}$-equivariant
associative ring orthogonal spectra to $S^{1}$-equivariant orthogonal
spectra. 
\end{prop}

A more careful analysis of the equivalence relation forming 
$\bar{\mathrm C}^{c}\otimes_{\bT}R$ from $\bar{\mathrm C}^{c}R$ should
show that the natural transformation is an isomorphism, but we take a
different approach.  Consider the case when $R=\bT X$ for some
$C_{m}$-equivariant orthogonal spectrum $X$.  Then the inclusion of
$X$ in $\bT X$ induces an isomorphism 
\[
\bar{\mathrm C}^{c}X\overto{\iso} \bar{\mathrm C}^{c}\otimes_{\bT}\bT X.
\]
In this case, the $m$-cyclic object $N^{cyc,C_{m}}_{\sma,\bullet}(\bT
X)$ breaks up into a wedge sum of $m$-cyclic objects
\[
N^{cyc,C_{m}}_{\sma,\bullet}(\bT X)= N(0)\subdot \vee N(1)\subdot \vee \dotsb 
\]
where $N(n)\subdot$ consists of the $X^{(n)}$ summands in
$N^{cyc,C_{m}}_{\sma,\bullet}(\bT X)$.  Then $N(0)$ is the constant
$m$-cyclic object on $\bS$, and for $n>0$, the inclusion of $X^{(n)}$
in $(\bT X)^{(n)}=N^{cyc,C_{m}}_{\sma,n-1}(\bT X)$ induces an
isomorphism of $C_{m}$-cyclic objects
\[
N(n)\subdot \overfrom{\iso} (\Lambda_{m}[n-1]\subdot)_{+}\sma_{C_{mn}} X^{(n)}.
\]
The map
\[
\bar{\mathrm C}^{c}X\to \bar{\mathrm C}^{c}\otimes_{\bT}\bT X \to
N^{cyc,C_{m}}_{\sma}(\bT X)
\]
respects homogeneous degree and induces an isomorphism 
\[
uC^{c}(n)_{+}\sma_{Z_{n}\wr C_{m}}X^{(n)}\to \Lambda_{m}[n-1]_{+}\sma_{C_{mn}} X^{(n)}=N(n)
\]
for each $n$.  This proves the following proposition.

\begin{prop}
Let $R=\bT X$ for $X$ a $C_{m}$-equivariant orthogonal spectrum.  The
natural $S^{1}$-equivariant map $\bar{\mathrm C}^{c}\otimes_{\bT}\bT
X\to N^{cyc,C_{m}}_{\sma}(\bT X)$ of Proposition~\ref{prop:CtoL} is an
isomorphism. 
\end{prop}

We can now prove Theorem~\ref{thm:cycbar} (which completed the proof
of Theorem~\ref{thm:circle}).

\begin{proof}[Proof of Theorem~\ref{thm:cycbar}]
Proposition~\ref{prop:CtoL} constructs the natural transformation.
Because smash powers preserve reflexive coequalizers, both
$\bar{\mathrm C}^{c}\otimes_{\bT}(-)$ and 
$N^{cyc,C_{m}}_{\sma}(-)$ preserve reflexive coequalizers.  Any
$C_{m}$-equivariant associative ring orthogonal spectrum is the
reflexive coequalizer
\[
\xymatrix@C-.75pc{%
\bT\bT R\ar@<-.5ex>[r]\ar@<+.5ex>[r]
&\bT R\ar[r]
&R
}
\]
where the maps $\bT\bT R\to \bT R$ are given by the monadic product
$\bT\bT\to \bT$ and $\bT$ applied to the action map $\bT R\to R$.
The previous proposition implies that the natural transformation is an
isomorphism for $\bT\bT R$ and $\bT R$, and so we conclude that it is
an isomorphism for $R$.
\end{proof}

\section{Factorization homology of $V$-framed $G$-manifolds
(equivariant theory)}\label{sec:efh}

In this section, we show how to obtain a version of $G$-equivariant factorization homology of $V$-framed
$G$-manifolds as a special case of the factorization homology
discussed in Section~\ref{sec:fhne}.  We start with a review of genuine
equivariant factorization homology for $V$-framed $G$-manifolds.  We
then compare the category of $V$-framed $G$-manifolds to a category of
(rigid) $G$-objects in a category of $G$-framed embeddings as defined
in Section~\ref{sec:fhne}.  Finally, we compare the bar construction
defining factorization homology of $V$-framed $G$-manifolds with the
bar construction of Construction~\ref{cons:compressed}.
Convention~\ref{conv:manifold} is in effect in this section.

We begin by reviewing the definition of $V$-framed $G$-manifolds and
their factorization homology as defined by
Horev~\cite{Horev-EqFactHom} and Zou~\cite{Zou-EqFactHom} in the case
of a finite group $G$. We follow the latter precisely, but work in the
context of $G$-equivariant orthogonal spectra rather than $G$-spaces
(see also Remark~\ref{rem:spaces} below); we expect that this agrees
with the genuine equivariant factorization homology of the former
(which has an axiomatic characterization) but make no justification
for that here.

\begin{defn}\label{defn:Vframe}
Let $V$ be finite dimensional vector space with linear $G$-action and
$G$-invariant inner product.  A $V$-framed $G$-manifold $M$ is a
smooth manifold with smooth $G$-action together with an isomorphism of
$G$-equivariant vector bundles
\[
\theta_{M}\colon TM \iso M\times V.
\]
We write $\theta_{M,V}\colon TM\to V$ for the composite of
$\theta_{M}$ with the projection to $V$.
Given $V$-framed $G$-manifolds $L$ and $M$, a $V$-framed embedding
$L\to M$ consists of a smooth embedding $f\colon L\to M$, a map
\[
\alpha \colon TL\times [0,\infty)\to V
\]
that is a linear isomorphism $T_{x}L\to V$ on each fiber,
and a locally constant function $\ell\colon 
L\to [0,\infty)$ such that
\begin{itemize}
\item for all $\xi\in TL$, $\alpha(\xi,0)=\theta_{L,V}(\xi)$; and
\item for all $\xi\in TL_{x}$ and $t\geq \ell(x)$, $\alpha(\xi,t)=\theta_{M,V}(Df(\xi))$.
\end{itemize}
We write $\EFV$ for the $G$-space of such maps with its intrinsic
topology and conjugation $G$-action.  The composition of $V$-framed
embeddings
\[
\EFV(L,N)\times \EFV(M,N)\to \EFV(L,M)
\]
is defined by treating $\ell$ as the length for Moore path
composition: given $(f,\alpha,\ell)\in \EFV(L,M)$ and
$(f',\alpha',\ell')\in \EFV(M,N)$, the composite map
$(f'',\alpha'',\ell'')\in  \EFV(L,N)$ is defined by
\begin{gather*}
f''=f'\circ f, \qquad \ell''=\ell'+\ell\\
\alpha''(\xi,t)=\begin{cases}
\alpha(\xi,t)&0\leq t\leq \ell(x)\\
\alpha'(Df(\xi),t-\ell(x))&t\geq \ell(x)
\end{cases}
\end{gather*}
(where $\xi\in T_{x}L$).  This makes $\EFV$ a $G$-topological
category (morphisms are $G$-spaces, composition is $G$-equivariant,
identity maps are $G$-fixed points) with the identity map in
$\EFV(M,M)$ given by $(\id,\alpha,0)$ with
$\alpha(\xi,t)=\theta_{M,V}(\xi)$. 
\end{defn}

The topological category agrees with the Definition~3.6 of
\cite{Zou-EqFactHom} (with the minor correction to $\ell$ to make
Remark~3.9, \textit{ibid.} work).  For $n\geq 0$, let $V(n)=V\times \{1,\dotsc,n\}$
with $V(0)$ the empty set.  We note that $V(n)$ has a canonical
$V$-framed $G$-manifold structure with $\theta$ the usual
identification $TV\iso V\times V$.

\begin{notn}
For $n\geq 0$, let $\oR_{M}(n)=\EFV(V(n),M)$, and let $\bar{\mathrm
R}_{M}$ denote the functor from $G$-equivariant orthogonal spectra to
$G$-equivariant orthogonal spectra defined by
\[
\bar{\mathrm R}_{M}X=\bigvee_{n\geq 0}\oR_{M}(n)_{+}\sma_{\Sigma_{n}}X^{(n)}
\]
with the diagonal $G$-action.  
\end{notn}

In the case when $M=V$, composition gives $\oR_{V}$ the structure of
an operad in $G$-spaces, and $\bar{\mathrm R}_{V}$ is the free
$\oR_{V}$-algebra monad on $G$-equivariant orthogonal spectra.  We
then get a right $\bar{\mathrm R}_{V}$ action on $\bar{\mathrm
R}_{M}$, and we can form the monadic bar construction.

\begin{cons}\label{cons:gefh}
For a $V$-framed $G$-manifold $M$ and a $\oR_{V}$-algebra in
$G$-equi\-variant orthogonal spectra $X$, define
\[
\bar B_{V}(M;X):=B(\bar{\mathrm R}_{M},\bar{\mathrm R}_{V},X).
\]
For $X$ an $\oR_{V}$-algebra in $G$-equivariant orthogonal spectra
indexed on a complete universe $U$, define $G$-equivariant
factorization homology by
\[
\int_{M}X:=I_{\bR^{\infty}}^{U}\bar B_{V}(M;I^{\bR^{\infty}}_{U}X).
\]
\end{cons}

\begin{rem}\label{rem:spaces}
The analogue for spaces obviously does not use point-set change of
universe, and should take 
\[
\bar{\mathrm R}_{M}X=\coprod_{n\geq 0} \oR_{M}(n)\times_{\Sigma_{n}}X^{(n)}.
\]
Zou~\cite[3.14]{Zou-EqFactHom} uses the reduced constructions that
glue along the inclusion of units $X^{n}\to X^{n+1}$ and the operad
degeneracies (operadic composition with $\oR_{V}(0)$).  A
straight-forward Quillen~A argument shows that the bar construction
for the reduced and the unreduced functors are $G$-equivariantly
homotopy equivalent.
\end{rem}

To compare this to the theory of Section~\ref{sec:fhne}, we note that
Definition~\ref{defn:Vframe} uses path composition of homotopies
whereas Definition~\ref{defn:framemap} uses pointwise multiplication
of homotopies.  This is easy to fix by enlarging $\EFV$ to use both.

\begin{defn}\label{defn:EFVBox}
Let $\EFV[\Box]$ be the $G$-topological category whose objects are
the $V$-framed $G$-manifolds and where a map from $L$ to $M$ consists
of a smooth embedding $f\colon L\to M$, a map
\[
\alpha \colon TL\times [0,\infty)\times [0,1]\to V
\]
that is a linear isomorphism $T_{x}L\to V$ in each fiber, and a locally constant function $\ell\colon
L\to [0,\infty)$ such that
\begin{itemize}
\item for all $\xi\in TL$, $\alpha(\xi,0,0)=\theta_{L,V}(\xi)$;
\item for all $\xi\in TL_{x}$, $\alpha(\xi,\ell(x),1)=\theta_{M,V}(Df(\xi))$;
\item for all $\xi\in TL_{x}$, $s\geq \ell(x)$, and $t\in [0,1]$, $\alpha(\xi,s,t)=\alpha(\xi,\ell(x),t)$.
\end{itemize}
We topologize $\EFV[\Box]$ with its intrinsic
topology and use conjugation for the $G$-action.  The composition 
\[
\EFV[\Box](M,N)\times \EFV[\Box](L,M)\to \EFV[\Box](L,N)
\]
is defined by treating $\ell$ as the length for Moore path
composition and doing pointwise multiplication (of linear isomorphisms) on
the overlap: given 
$(f,\alpha,\ell)\in \EFV[\Box](L,M)$ and 
$(f',\alpha',\ell')\in \EFV[\Box](M,N)$, the composite map
$(f'',\alpha'',\ell'')\in  \EFV[\Box](L,N)$ is given by
\begin{gather*}
f''=f'\circ f, \qquad \ell''=\ell'+\ell\\
\alpha''(\xi,s,t)=\begin{cases}
\alpha'(\tilde\alpha(\xi,s,t),0,t)&0\leq s\leq \ell(x)\\
\alpha'(\tilde\alpha(\xi,s,t),s-\ell'(x),t)&s\geq \ell(x)
\end{cases}
\end{gather*}
where $\xi\in TL_{x}$ and $\tilde\alpha\colon TL\times [0,1]\times
[0,\infty)\to TM$ is defined by
\[
\theta_{M}(\tilde\alpha(\xi,s,t))=(f(x),\alpha(\xi,s,t))\in M\times V.
\]
\end{defn}

With this definition $\EFV$ includes (isomorphically) as the
subcategory $\EFV[-]$ of maps where $\alpha$ is constant in the $[0,1]$
direction.  The inclusion of $\EFV[-](L,M)$ in
$\EFV[\Box](L,M)$ is always a $G$-equivariant homotopy equivalence.
We can define an analogue of $\oR^{\Box}_{M}$ of $\oR_{M}$
using $\EFV[\Box]$ in place of $\EFV$.  With this, we get a corresponding
monadic bar construction $B^{\Box}_{V}(M;-)$, and for any
$G$-equivariant $\oR^{\Box}_{V}$-algebra $X$, we get a map
\begin{equation}\label{eq:efhv1}
\bar B_{V}(M;X)\to \bar B^{\Box}_{V}(M;X)
\end{equation}
which is evidently a natural $G$-equivariant homotopy equivalence. Let
$\EFV[|]$ be the subcategory where $\ell=0$; the inclusion of $\EFV[|](L,M)$ in
$\EFV[\Box](L,M)$ is always a $G$-equivariant homotopy equivalence.  We 
get an analogue $\oR^{|}_{M}$ of $\oR^{\Box}_{M}$
and a corresponding monadic bar construction $B^{|}_{V}(M;-)$; for 
any $G$-equivariant $\oR^{\Box}_{V}$-algebra $X$, the induced map
\begin{equation}\label{eq:efhv2}
\bar B^{|}_{V}(M;X)\to \bar B^{\Box}_{V}(M;X)
\end{equation}
is a natural $G$-equivariant homotopy equivalence.  It is therefore
harmless to use $\EFV[|]$ in place of $\EFV$.

In the context of Section~\ref{sec:fhne}, we will take $H=G$, but the
group $G$ will play two different roles here, and will retain the
notation $H$ for $G$ as the structure group as in that section when
needed to
avoid confusion between the different roles. (The following definition
still makes sense for $H\neq G$ provided it is $H$ that acts on $V$.)
We are then considering the category $\aE_{H,V}$ of $H$-framed
embeddings of manifolds with tangential $H,V$-structure as in
Definition~\ref{defn:framemap} and we define a $G$-equivariant
$H$-framed $G$-manifold to consist of an $H$-framed manifold $M$
together with an action of $G$ on $M$ in $\aE_{H,V}$ where each element
$g$ of $G$ acts by an $H$-framed local isometry.  We view this as a
$G$-topological category with $G$-space of maps given by the space of
maps in $\aE_{H,V}$ with $G$ acting by conjugation; we denote this
category as $\aE_{H,V;G}$. 

A $V$-framed $G$-manifold $M$ gives an object in $\aE_{G,V;G}$ as follows:
we take the $G=H$-frame bundle to be the projection map $M\times G\to
M$, and we then have an isomorphism $F_{G}M\times_{G}V\to TM$ 
\[
F_{G}M\times_{G}V=(M\times G)\times_{G}V \iso M\times V\overto{\theta_{M}^{-1}}TM
\]
which induces an isomorphism of $\GL(V)$-principal bundles from
$F_{G}M\times_{G}\GL(V)$ to the $V$-frame bundle of $M$.  (More
concretely, the adjoint to $\theta^{-1}_{M}$ gives a section of $M$ into
the $V$-frame bundle, and the $G$-action on $V$ then gives the
reduction of structure to $M\times G$.)  This defines
the tangential $G,V$-structure.  Since for $g\in G$, the diagram
\[
\xymatrix{%
TM\ar[r]^{Dg}\ar[d]_{\theta_{M}}&TM\ar[d]^{\theta_{M}}\\
M\times V\ar[r]_{(g,g)}&M\times V
}
\]
commutes, the diagonal action of $G$ on $M\times G$ gives a map of
$G=H$-frame bundles lifting the derivative on the $V$-frame bundle. In
other words, lifting $g$ to the diagonal $g$-action on $F_{G}M$ endows
$g$ with the structure of a $G=H$-framed local isometry. 

The resulting object of $\aE_{G,V;G}$ comes with the extra structure
of a section $s_{M}$ of the $G$-frame bundle, namely the section at
the identity element of $G$. This section is compatible with the group action
on $M$ (in $\aE_{G,V;G}$) and on $V$ in the sense that for all $x\in
M$ and all $g\in G$, the equation 
\[
g\cdot s_{M}(x) = s_{M}(gx) \cdot g \in (F_{G}M)_{gx},
\]
where the action on the lefthand side denotes the action of $g$ on the
$H$-frame bundle of $M$ as an $H$-framed local isometry and the action
on the righthand side denotes the right action of $g\in G$ on the
$G=H$-principal bundle.  (Using the section to identify the $G$-frames
at each point as $G$, this equation reads $g\cdot e=e\cdot g\in G$.) For
an arbitrary object $M$ of $\aE_{G,V;G}$, we call a section $s_{M}$ of
the $G$-frame bundle satisfying the equation above
\term{$G,V$-compatible}.  We can now state the precise relationship between the
categories $\EFV[|]$ and $\aE_{G,V;G}$.

\begin{thm}\label{thm:EFV}
The $G$-topological category $\EFV[|]$ is equivalent to the
$G$-topological category $\aP$ where an objects is an ordered pair $(M,s_{M})$
with $M$ an object of $\aE_{G,V;G}$ and $s_{M}$ is a
$G,V$-compatible section of its $G$-frame bundle and where the $G$-space
of maps $(L,s_{L})\to (M,s_{M})$ is the subspace of $\aE_{G,V}(L,M)$
of maps whose induced map of $G$-frame bundles takes $s_{L}$ to
the pullback of $s_{M}$.
\end{thm}

\begin{proof}
We have already described the functor in the forward direction on
objects.  On maps, it takes $(f,\alpha,0)\in \EFV[|](L,M)$ as above to the
map $(f,Ff,If)\in \aE_{G,V}(L,M)$ defined as follows.  The map
\[
Ff\colon F_{G}L=L\times G\to f^{*}F_{G}M\iso L\times G
\]
is $f\times \id\colon L\times G\to L\times G$, which takes
the section $s_{L}$ of $F_{G}L$ to the section $f^{*}s_{M}$ of
$f^{*}F_{G}M$ as required.  We note that for $x\in L$, the element
$s_{L}(x)\in (F_{G}L)_{x}$ maps to the $V$-frame $\theta^{-1}_{L,x}$ in
$(F_{V}L)_{x}=\Iso(V,T_{x}L)$, where
$\theta^{-1}_{L,x}\colon V\iso T_{x}L$ denotes the restriction of the
bundle map $\theta_{L}^{-1}\colon L\times V\iso TL$ to the point $x$.
The $\GL(V)$-principal bundle homotopy $If$ needs to start at the
$\GL(V)$-principal bundle map that sends 
$\theta^{-1}_{L,x}$ to $\theta^{-1}_{M,f(x)}$ and end at the
map that takes $\theta^{-1}_{L,x}$ to $Df_{x}\circ \theta^{-1}_{L,x}$.
Define $If$ fiberwise by
\begin{gather*}
If_{x}\colon \Iso(V,T_{x}L)\times [0,1]\to \Iso(V,T_{f(x)}M)\\
If_{x}(\phi,t)=\theta_{M,f(x)}^{-1}\circ \alpha((x,-),0,t)\circ \phi,
\end{gather*}
where we understand $\alpha((x,-),0,t)$ as a linear isomorphism
$T_{x}L\to V$ in the composition formula.  The formula is clearly
right $\GL(V)$-equivariant and specifies a homotopy of maps of
$\GL(V)$-principal bundles. 
Then since $\alpha((x,-),0,0)=\theta_{L,x}$, we have
\[
If_{x}(\theta_{L,x}^{-1},0)=\theta_{M,f(x)}^{-1}\circ \theta_{L,x}\circ \theta^{-1}_{L,x}=\theta^{-1}_{M,f(x)},
\]
and since $\alpha((x,-),0,1)=\theta_{M,x}\circ Df_{x}$, we have 
\[
If_{x}(\theta_{L,x}^{-1},1)=\theta^{-1}_{M,f(x)}\circ \theta_{M,x}\circ Df_{x}\circ \theta_{L,x}^{-1}
=Df_{x}\circ \theta_{L,x}^{-1},
\]
as required.

Now given $M,s_{M}$, we get a $V$-framed $G$-manifold structure on $M$
using the given $G$ action on $M$ and the isomorphism
$\theta^{-1}_{M}\colon M\times V\iso TM$ adjoint to the section $s_{M}$;
the $G,V$-compatibility precisely implies that this isomorphism is
$G$-equivariant. Given $(L,s_{L})$, $(M,s_{M})$, and a map
$(f,Ff,If)\in \aE_{G,V}(L,M)$ that on $G$-frame bundles sends $s_{L}$
to $f^{*}s_{M}$, we produce a map $(f,\alpha,0)\in \EFV[|](L,M)$,
by defining
\[
\alpha((x,-),0,t)=\theta_{M,f(x)}\circ If_{x}(\theta^{-1}_{L,x},t)\circ \theta_{L,x}.
\]
Since by hypothesis $Ff$ sends $s_{L}$ to $f^{*}s_{M}$, we have that 
$If_{x}(\theta^{-1}_{L,x},0)=\theta^{-1}_{M,f(x)}$, and we see that
$\alpha((x,-),0,0)=\theta_{L,x}$, as required.  Since by definition 
$If_{x}(\theta^{-1}_{L,x},1)=Df_{x}\circ \theta^{-1}_{L,x}$, we see
that $\alpha((x,-),0,1)=\theta_{M,f(x)}Df_{x}$, as also required.

It is straight-forward to check that these formulas define functors,
that the composite functor on $\EFV[|]$ is the identity, and that the
composite functor on the pair category is naturally isomorphic to the
identity.
\end{proof}

We note that both $\EFV[|]$ and $\aE_{G,V;G}$ have a coproduct, given
on the underlying $G$-manifolds by disjoint union, and that the
functor in the previous definition preserves the coproduct.

For a $V$-framed $G$-manifolds, the main difference (philosophically)
between maps in $\EFV[|]$ and maps in $\aE_{G,V;G}$ is that for $g\in
G$, the self-map $g\colon M\to M$ in $\aE_{G,V;G}$ is never in the
image of the maps in $\EFV[|]$ unless $g=e$ (even if $G$ acts
trivially on both the vector space $V$ and the underlying smooth
manifold of $M$).  We can put this action back in with the following
extension of $\EFV[|]$. 

\begin{defn}
Let $\EFV[|\rtimes G]$ be the topological category where the objects are
the $V$-framed $G$-manifolds and where the maps are defined by
\[
\EFV[|\rtimes G](L,M)=\EFV[|](L,M)\times G^{\pi_{0}L},
\]
where $G^{\pi_{0}L}$ denotes the space of locally constant maps from
$L$ to $G$.  Using the observation about coproducts above, writing
$L=L_{1}\amalg \dotsb \amalg L_{p}$ for the components of $L$, we have
\[
\EFV[|\rtimes G](L,M)\iso \EFV[|\rtimes G](L_{1},M)\times \dotsb \times \EFV[|\rtimes G](L_{p},M);
\]
moreover, writing $M=M_{1}\amalg \dotsb \amalg M_{q}$ for the
components of $M$, we also have
\[
\EFV[|\rtimes G](L_{i},M)\iso \EFV[|\rtimes G](L_{i},M_{1})\amalg
\dotsb \amalg \EFV[|\rtimes G](L_{i},M_{q}),
\]
and so it suffices to describe compositions 
\[
\EFV[|\rtimes G](M,N)\times \EFV[|\rtimes G](L,M)\to \EFV[|\rtimes G](L,N)
\]
when all of $L$, $M$, and $N$ are connected.  For 
\begin{gather*}
(\phi,g)\in \EFV[|](L,M)\times G=\EFV[|\rtimes G](L,M),\\
(\phi',g')\in \EFV[|](M,N)\times G=\EFV[|\rtimes G](L,M),
\end{gather*}
the composite in $\EFV[|\rtimes G](L,N)$ is given by $(\phi'\circ
{{}^{g'}\!\phi},g'g)$ where the superscript $g'$ denotes the $G$-action on
$\EFV[|](L,M)$. 
\end{defn}

We then have a functor $\EFV[|\rtimes G]\to \aE_{G,V;G}$ which sends
$(\phi,g)\in \EFV[|\rtimes G](L,M)$ to the composite of the self-map
of $L$ given by $g$ and the image of the map $\phi$ under the functor
of Theorem~\ref{thm:EFV}.  Composition in $\EFV[|\rtimes G]$ is
defined precisely to make this functorial. For $V$-framed
$G$-manifolds $L,M$, every map in $\aE_{G,V;G}(L,M)$ comes with an
associated continuous map $L\to G$ given by the inherent map on $G$-frame bundles
and canonical sections; when $G$ is finite the map is locally constant
and decomposes the element of $\aE_{G,V;G}(L,M)$ as an element of
$\EFV[|\rtimes G]$.  This proves the following, which we regard as a
corollary of Theorem~\ref{thm:EFV}.

\begin{cor}\label{cor:EFV}
If $G$ is a finite group, then the topologically enriched functor
$\EFV[|\rtimes G]\to \aE_{G,V;G}$ is full and faithful: for all $L,M$
in $\EFV[|\rtimes G]$, the map 
\[
\EFV[|\rtimes G](L,M)\to \aE_{G,V;G}(L,M)
\]
is a homeomorphism.
\end{cor}

The relationship is not so tight when $G$ is a positive dimensional
compact Lie group; however, for the purposes of factorization
homology, it suffices to understand the relationship for the disjoint
union of copies of $V$.  In that case we can study
$\aE_{G,V;G}(V(n),M)$ in terms of configurations; looking at the
center point of the disk (in $M$ and in the $G$-frame bundle map),
we get a commutative diagram
\[
\xymatrix@C-3pc{%
\EFV[|\rtimes G](V(n),M)\ar[dr]\ar[rr]&&\aE_{G,V;G}(V(n),M)\ar[dl]\\
&C(n,M)\times G^{n}
}
\]
and well-known arguments show that the downward maps are homotopy
equivalences.  Keeping track of equivariance, we have a left
$G$-action on $M$ and a left $\Sigma_{n}\wr G$ action on $V(n)$ that
the mapping spaces converts to a right $\Sigma_{n}\wr G$-action, which
is compatible with the right $\Sigma_{n}\wr G$ action on $G^{n}$.  An
equivariant elaboration of the usual argument shows that the maps are
in fact equivariant homotopy equivalences: 

\begin{prop}
For an arbitrary compact Lie group $G$, and a $V$-framed $G$-manifold
$M$, the map
\[
\EFV[|\rtimes G](V(n),M)\to \aE_{G,V;G}(V(n),M)
\]
is a $G\times (\Sigma_{n}\wr G)^{\op}$-equivariant homotopy equivalence.
\end{prop}

Unlike in the category $\EFV$, in the category $\EFV[|]$, the vector
space $V$ and its unit disk $D$ are isomorphic:
choosing a smooth diffeomorphism $\usc\colon [0,1)\to [0,\infty)$ that is the
identity near 0, we can use $\usc$ radially to get a diffeomorphism of
$G$-manifolds $\USC\colon D\to V$. We take $\alpha$ to be 
\[
\alpha((x,v),0,t)=(\usc'(|x|))^{t}v_{x}+(\usc(|x|)/|x|)^{t}v_{\perp x}
\]
where $v_{x}$ denotes orthogonal projection in the $x$ direction and
$v_{\perp x}$ denotes the orthogonal complement; for fixed $x,v$, this is a
homotopy from the identity to the derivative of $\USC$.  (For points
$x$ near $0$, we have $\usc'(|x|)=1$ and $\usc(|x|)/|x|=1$, so
$\alpha$ is the constant homotopy, and we understand the formula this
way also at the point $x=0$.)

Writing $\oR^{|D}_{M}(n)=\EFV[|](D(n),M)$ for $D(n)=D\times
\{1,\dotsc,n\}$ and
\[
\bar{\mathrm R}^{|D}_{M}X=\bigvee_{n\geq 0}\oR^{|D}_{M}(n)_{+}\sma_{\Sigma_{n}}X^{(n)}
\]
(with the diagonal $G$-action), we then get a monadic bar construction
\[
\bar B^{|D}_{V}(M;-):=B(\bar{\mathrm R}^{|D}_{M},\bar{\mathrm R}^{|D}_{D},-),
\]  
and the isomorphism $\USC$ above induces a $G$-equivariant isomorphism 
\begin{equation}\label{eq:efhd1}
\bar B^{|D}_{V}(M;-)\iso \bar B^{|}_{V}(M;-)
\end{equation}
where we also use $\USC$ to translate between the inputs.  

The little $V$-disk operad $\oD_{V}$ admits an obvious map of
$G$-equivariant operads $\oD_{V}\to \oR^{|D}_{D}$, where we interpret
an affine transformation $\lambda(v)=v_{0}+rv$ as the map in $\EFV[|]$
given by $\lambda$ and the homotopy $\alpha((x,v),0,t)=r^{t}v$.
Looking at configuration spaces, we see that the maps 
\[
\oD_{V}(n)\to \oR^{|D}_{D}(n)
\]
are $G\times \Sigma_{n}^{\op}$-equivariant homotopy
equivalences, and so the induced map on monadic bar constructions
\begin{equation}\label{eq:efhd2}
B(\bar{\mathrm R}^{|D}_{M},\bar\bD,-)\to
B(\bar{\mathrm R}^{|D}_{M},\bar{\mathrm R}^{|D}_{D},-)=\bar B^{|D}_{V}(M;-)
\end{equation}
is a natural $G$-equivarant homotopy equivalence. 

To compare the monadic bar construction $B(\bar{\mathrm
R}^{|D}_{M},\bar\bD,-)$ to the monadic bar construction 
\[
\bar B(M;-)=B(\bar\bE_{M},\bar\bD,-)
\]
of Section~\ref{sec:fhne}, we observe that as functors from
$G$-equivariant orthogonal spectra to itself, the functor
\[
\bar{\mathrm R}^{|D}_{M}=\bigvee_{n\geq 0}\EFV[|](D(n),M)_{+}\sma_{\Sigma_{n}}X^{(n)}
\]
with the diagonal $G$-action is naturally isomorphic to the functor
\[
\bigvee_{n\geq 0}\EFV[|\rtimes G](D(n),M)_{+}\sma_{\Sigma_{n}\wr G}X^{(n)}
\]
with the $G$-action given by the left $G$-action on $\EFV[|\rtimes G](D(n),M)$.
The functor in Corollary~\ref{cor:EFV} then gives a natural transformation 
\[
\bar{\mathrm R}^{|D}_{M}(X)\to \bar \bE_{M}(X)=
\bigvee_{n\geq 0}\aE_{G,V}(D(n),M)_{+}\sma_{\Sigma_{n}\wr G}X^{(n)},
\]
which is a natural $G$-equivariant homotopy equivalence (an
isomorphism when $G$ is finite).  Moreover,
the natural transformation $\bar{\mathrm R}^{|D}_{M}\to \bar\bE_{M}$
is a map of right $\bar\bD$-functors, and so we get an induced map on
bar constructions that is also a natural $G$-equivariant homotopy
equivalence:

\begin{prop}\label{prop:efhd3}
For any $G$-equivariant $\oD_{V}$-algebra $X$, the natural map
\[
B(\bar{\mathrm R}^{|D}_{M},\bar\bD,X)\to 
B(\bar\bE_{M},\bar\bD,X)=\bar B(M;X)
\]
is a natural $G$-equivariant homotopy equivalence.
\end{prop}

All together, \eqref{eq:efhv1}, \eqref{eq:efhv2}, \eqref{eq:efhd1},
\eqref{eq:efhd2}, and Proposition~\ref{prop:efhd3} give a zigzag of
natural $G$-equivariant homotopy equivalences between genuine
equivariant factorization homology for $V$-framed $G$-manifolds as
defined in Definition~\ref{cons:gefh} (inspired by
Horev~\cite{Horev-EqFactHom} and Zou~\cite{Zou-EqFactHom})  and the
theory 
\[
I_{\bR^{\infty}}^{U}\bar B(M;I^{\bR^{\infty}}_{U}X) 
\]
of Section~\ref{sec:fhne}.


\bibliographystyle{plain}
\begingroup
   \def\refname{Missing Pieces}

\endgroup

\bibliography{bluman}

\end{document}